\documentclass[a4paper, 11pt]{article}

\input xy
\xyoption{all}

\usepackage{amsmath,amsthm,amssymb}
\usepackage[page,header]{appendix}
\usepackage{morefloats}
\usepackage{color}
\usepackage{makeidx}
\usepackage{hyperref}	 	
\usepackage{fancybox}
\usepackage{mathabx}
\usepackage{tikz}
\usetikzlibrary{matrix}
\usepackage{arcs}
\usepackage{hyperref}

\makeindex 

\usepackage{latexsym, amsmath, amssymb, amsfonts, amscd}
\usepackage{amsthm}
\usepackage{t1enc}
\usepackage[mathscr]{eucal}
\usepackage{indentfirst}
\usepackage{graphicx, pb-diagram}
\usepackage{enumerate}
\usepackage[all,poly,web,knot]{xy}
\usepackage{float}
\restylefloat{figure}

\newcommand{\Title}{Title}

\numberwithin{equation}{section}

{\theoremstyle{definition}\newtheorem{definition}{Definition}[section]

\newtheorem{defititle}[definition]{\Title}
\newtheorem{notation}[definition]{Notation}

\newtheorem{remark}[definition]{Remark}
\newtheorem{remarks}[definition]{Remarks}
\newtheorem{ex}[definition]{Example}
\newtheorem{exs}[definition]{Examples}}
\newtheorem{prop}[definition]{Proposition}
\newtheorem{proposition-definition}[definition]{Proposition-Definition}
\newtheorem{lemma}[definition]{Lemma}
\newtheorem{thm}[definition]{Theorem}
\newtheorem{cor}[definition]{Corollary}

\newtheorem*{prop*}{Proposition}
\newtheorem*{theorem*}{Theorem}

\newcommand{\cB}{\mathcal{B}}
\newcommand{\cD}{\mathcal{D}}

\newcommand{\cF}{\mathcal{F}}

\newcommand{\cR}{\mathcal{R}}

\newcommand{\cU}{\mathcal{U}}

\newcommand{\cH}{\mathcal{H}}

\newcommand{\cf}{{\it cf.}\/ }

\def\gpd{\,\lower1pt\hbox{$\longrightarrow$}\hskip-.24in\raise2pt
             \hbox{$\longrightarrow$}\,}

\renewcommand{\latticebody}{\drop@{ }}

\newcommand{\N}{\ensuremath{\mathbb N}}

\newcommand{\C}{\ensuremath{\mathbb C}}
\newcommand{\R}{\ensuremath{\mathbb R}}

\newcommand{\g}{\ensuremath{\mathfrak{g}}}

\newcommand{\cX}{\mathcal{X}}

\newcommand{\RR}{\ensuremath{\mathbb R}}

\def\act{\mathbin{\hbox{$<\kern-.4em\mapstochar\kern.4em$}}}
\def\ract{\mathbin{\hbox{$\mapstochar\kern-.3em>$}}}

\def\PB(#1,#2,#3,#4){\left\{\begin{matrix}#1&\!\!\!\stackrel{?}{\longrightarrow}&\!\!\!#2\\
\downarrow&&\!\!\!\downarrow\\
#3&\!\!\!\stackrel{?}{\longrightarrow}&\!\!\!#4\end{matrix}\right\}}

\def\pb(#1,#2,#3,#4){ \hom(#1 \to #3, #2 \to #4)}

\begin{document}

\begin{center}
{\Large\bf Riemannian metrics and Laplacians for generalised smooth distributions\footnote{AMS subject classification: 58J60 ~ Secondary 53C17, 58A30, 35H10, 35R01. Keywords: vector distribution, singular foliation,  Riemannian structures, Laplacian, differential operators, subelliptic estimates, hypoellipticity.}}


\bigskip

{\sc by Iakovos Androulidakis and
 Yuri Kordyukov\footnote{The second named author was supported by the Russian Foundation of Basic Research, grant no. 16-01-00312.} 
}
 
\end{center}

{\footnotesize
Department of Mathematics 
\vskip -4pt National and Kapodistrian University of Athens
\vskip -4pt Panepistimiopolis
\vskip -4pt GR-15784 Athens, Greece
\vskip -4pt e-mail: \texttt{iandroul@math.uoa.gr}

\vskip 2pt Institute of Mathematics  
\vskip-4pt Ufa Federal Research Centre
\vskip-4pt Russian Academy of Science 
\vskip-4pt 112 Chernyshevsky str.
\vskip-4pt 450008 Ufa, Russia
\vskip-4pt e-mail: \texttt{yurikor@matem.anrb.ru}

}
\bigskip
\everymath={\displaystyle}

\date{today}

\begin{abstract}\noindent 
We show that any generalised smooth distribution on a smooth manifold, possibly of non-constant rank, admits a Riemannian metric. Using such a metric, we attach a Laplace operator to any smooth distribution as such. When the underlying manifold is compact, we show that it is essentially self-adjoint. Viewing this Laplacian in the longitudinal pseudodifferential calculus of the smallest singular foliation which includes the distribution, we prove hypoellipticity. 
\end{abstract}
 
\setcounter{tocdepth}{2} 

\addcontentsline{toc}{section}{Introduction}

\tableofcontents

\section*{Introduction}\label{sec:intro}


One way to define and study important geometric and topological invariants of a smooth manifold is by attaching a natural differential operator to it and studying its analytic invariants. Such differential operators usually arise geometrically, that is to say using an appropriate geometric structure on the manifold. A fundamental example of such an operator is the Laplace-Beltrami operator (or the Laplacian) of a Riemannian manifold. Once a geometric differential operator is introduced, its self-adjointness needs to be proven first, in order to set a well-posed unbounded operator in a Hilbert space with good spectral properties. Here it is usually essential that the operator is an \emph{elliptic} differential operator. It allows one to use methods and results of theory of elliptic partial differential operators such as, first of all, the existence of parametrix, elliptic estimates and elliptic regularity. An appropriate pseudodifferential calculus and the associated scale of Sobolev spaces plays an important role in these considerations. Such an approach was generalised to many other settings, for instance, to singular manifolds, dynamical systems and foliations. In this article we carry out the first steps for the study of the Laplacians on an arbitrary generalised smooth distribution.

Roughly, generalised smooth distributions are smooth assignments of vector subspaces $D_x$ of $T_x M$, for every $x \in M$. These subspaces are not required to have constant rank.
This class contains all the distributions arising in sub-Riemannian geometry, in particular the non-equiregular sub-Riemannian structures (this is thanks to the formulation in \cite{Agrachev97} \cite{Bellaiche} of sub-Riemannian structures as anchored vector bundles $\rho : E \to TM$) as well as singular foliations, that is, involutive generalised smooth distributions (\cf \cite{AS1}).

Much like \cite{AS1}, we view a (smooth) distribution on a manifold $M$ much more in terms of its dynamics. This means that we focus on the module of vector fields $\cD$ rather than the family of vector subspaces $D=\bigcup_{x \in M}D_x$ of the tangent bundle $TM$ (whose dimension is non-constant). Note that $\cD$  is a primitive of $D$: Indeed, each $D_x$ is the evaluation at $x$ of $\cD$. For instance, given a sub-Riemannian structure $\rho : E \to TM$, we have $\cD = \rho(\Gamma_c E)$\footnote{We restrict to compactly supported vector fields in order to exhibit our results in the easiest possible setting. Removing the compact support condition requires to work with sheaves (\cf \cite{AZ5}).}, while $D=\rho(E)$.

A pseudodifferential calculus for a singular foliation was introduced in \cite{AS2}. Moreover, a longitudinal Laplacian was attached to a foliation as such, albeit merely as a sum of squares rather than by the use of some Riemannian metric. The singularities of the foliation made it quite difficult to use such a metric in a smooth way. Nevertheless, the longitudinal Laplacians were proven to be self-adjoint and elliptic, the latter thanks to the involutivity property.

If we relax the involutivity hypothesis, we pass to the much larger category of generalised smooth distributions. In order to attach a Laplacian to such a distribution in a geometric way, it is necessary to have a Riemannian metric on such a pathological object. Assuming such a metric can be constructed, the self-adjointness of the associated Laplacian might be expected. One can also consider the Laplacian to be longitudinally elliptic along the distribution, but it is more essential in the case when the distribution is involutive, that is, it is a singular foliation, because then one can use the pseudodifferential calculus mentioned above. In the case of a  general distribution $\cD$, it is natural to consider the smallest (singular) foliation $\cU(\cD)$ which includes $\cD$. In favourable cases, this foliation is given by a kind of universal enveloping algebra of the given distribution, otherwise it is just the one whose leaves are the entire connected components of $M$.  One can use the longitudinal pseudodifferential calculus for $\cU(\cD)$. The operator is not longitudinally elliptic with respect to $\cU(\cD)$, but, locally, it can be considered as the sum of squares operator for a family of vector fields satisfying the bracket generating condition along the leaves of $\cU(\cD)$. 
Recall that the bracket generating condition has central importance in sub-Riemannian geometry and control theory (\cf for instance, \cite{Brockett}, \cite{Montgomery}). It is also the key to H\"{o}rmander's result on the hypoellipticity of the sum of squares operator arising from given vector fields $X_1,\ldots,X_k$. So we can expect the Laplacian associated with the distribution $\cD$ should be longitudinally hypoelliptic with respect to $\cU(\cD)$.

The above considerations were confirmed in \cite{YK1} and \cite{YK2}, where the horizontal Laplacian of a smooth constant rank distribution $(M,\cD)$ was introduced and studied. In this case, the module $\cD$ is projective and in view of the familiar Serre-Swan theorem, one may think of $\cD$ as the $C^{\infty}(M)$-module of sections of the vector subbundle $D$ of $TM$. 
This is quite a large class of distributions, for instance it includes all the regular foliations and the constant-rank sub-Riemannian manifolds. In \cite{YK1}, a Riemannian metric on $\cD$ is defined to be a smooth family of inner products in the fibers of $D$ and the associated Laplace operator $\Delta_{\cD}$ (denoted by $\Delta_{D}$ in \cite{YK1}) is introduced. Using the Chernoff self-adjointness criterion \cite{Chernoff}, it was shown that this Laplacian is essentially self-adjoint as an operator on $L^2(M)$. For the study of more elaborated analytic properties of $\Delta_{\cD}$, the longitudinal pseudodifferential calculus for singular foliations developed in \cite{AS1}, \cite{AS2} plays a crucial role in \cite{YK1}, \cite{YK2}. It turns out that the horizontal Laplacian $\Delta_{\cD}$ of the distribution $\cD$ constructed in \cite{YK1} lives in the longitudinal pseudodifferential calculus of the foliation $\cU(\cD)$ and satisfies subelliptic estimates and hypoellipticity property in the scale of longitudinal Sobolev spaces.

Here we manage to extend these results to an arbitrary generalised smooth distribution $(M,\cD)$. That is to say, without the constant rank assumption. 
The main difficulty here is that the non-constant rank prevents the use of smooth families of inner products in the classical sense, whence one first have to understand how to construct the horizontal Laplacian in a geometric way. 

\subsection*{Methods and results}

As mentioned above, we view a (smooth) distribution on a manifold $M$ much more in terms of its dynamics, that is, as the $C^\infty_c(M)$-module $\cD$ of vector fields tangent to the distribution, which is assumed to be locally finitely generated.
We introduce the fiber of the distribution $\cD$ at $x$ as a finite dimensional vector space $\cD_x = \cD/I_x\cD$, where $I_x = \{f \in C^{\infty}(M) : f(x)=0\}$ and define a Riemannian structure for $\cD$ as a family of inner products $\langle\ ,\ \rangle_x$ on $\cD_x$, depending smoothly on $x\in M$. Our first result is:

\textbf{Theorem A} \textit{Let $(M,\cD)$ be an arbitrary smooth distribution. There exists a Riemannian structure for $(M,\cD)$.}


The difficulty in proving the existence of a Riemannian structure as such, is that the dimensions of the ``fibers'' $\cD_x$ are not constant, actually they vary in a semicontinuous\footnote{The dimensions of their evaluations $D_x$ vary with the opposite semicontinuity.} way. In order to make sense of smoothness for the family of inner products $\{\langle\ ,\ \rangle_x\}_{x \in M}$ in an effective way, we introduce a weak notion of coordinates for the distribution $\cD$. It is inspired from the viewpoint of sub-Riemannian structures as anchored vector bundles in \cite{Agrachev97}, \cite{Bellaiche}. Specifically, since our distribution $(M,\cD)$ is locally finitely generated, locally it can be described from an anchored vector bundle. That is to say, for every point $x$ in $M$, there are a small neighborhood $U$ of $x$ in $M$ and an anchored vector bundle $\rho_U : E_U \to TM$ over $U$ so that $\cD\left|_U\right.=\rho_{U}(E_U)$. More specifically, if the restrictions of $X_1,\ldots,X_k \in \cD$ on $U$ generate the module $\cD\left|_U\right.$, then one can take $E_U$ to be just the trivial bundle $U \times \R^k$ and $\rho_U(y,\lambda_1,\ldots,\lambda_k)=\sum_{i=1}^k \lambda_i X_i(y)$. We call the data $(E_U,\rho_U)$ constructed in this specific way, a \textit{local presentation}\footnote{When $\cF$ is a singular foliation, it is easy to see that local presentations arise from \textit{bisubmersions} (\cf \cite{AS1}).} of $(M,\cD)$ at the point $x$. Of course there are lots of choices involved in the construction of a local presentation as such. We introduce an equivalence relation between local presentations, which amounts to the change of coordinates for $\cD$. The proof of Theorem A is possible because at the equivalence classes associated with this relation, the various choices we made disappear naturally.


Using such a smooth family of Riemannian metrics, together with a positive density $\mu$ on $M$, given \textit{any} smooth distribution $(M,\cD)$ as above, we are able to show the following:
\begin{enumerate}
\item There is a geometric construction of a ``horizontal'' Laplace operator $\Delta_{\cD}$ for \emph{any} smooth distribution $(M,\cD)$. This is a second order differential operator acting on $C^\infty(M)$.
\item Locally, $\Delta_{\cD}$ can be expressed as a sum of squares of (local) generators of the module of vector fields $\cD$.
\item The operator $\Delta_{\cD}$ fits into the following pseudodifferential calculi:
\begin{itemize}
\item The standard pseudodifferential calculus of $M$.
\item When the algebra $\cU(\cD) = [\cD,\ldots,[\cD,\cD]]$ is a (singular) foliation (\cf \cite{AS1}), then $\Delta_{\cD}$ fits in the associated longitudinal pseudodifferential calculus \cite{AS2}.
\end{itemize}
\item When $\cU(\cD)$ is a foliation, the operator $\Delta_{\cD}$, considered as an unbounded operator on $L^2(M,\mu)$, with domain $C^{\infty}(M)$, is the (trivial) representation of a certain unbounded multiplier of $C^{\ast}_r(\cU(\cD))$.
\end{enumerate}

Note that if $\cD$ arises from a sub-Riemannian structure, the bracket generating condition says that $\cU(\cD)$ is the entire algebra  $\cX_c(M)$ of compactly supported vector fields on $M$. Hence, in this case $\Delta_{\cD}$ just lives in the standard pseudodifferential calculus of $M$. Also, in the case of a constant rank distribution, $D$ is a vector sub-bundle of $TM$ and $\cD$ is the $C^{\infty}(M)$-module of its (compactly supported) sections. As we already mentioned, the familiar Serre-Swan theorem says that $\cD$ carries no extra information than the bundle $D$. Whence, in the case of a constant rank distribution, all the results we give here reduce to the ones in \cite{YK1}.

Next, we are interested in the questions of self-adjointness and hypoellipticity of $\Delta_{\cD}$. To this end, we restrict to the case where $M$ is a compact manifold. Adapting the proofs given in \cite{YK1} in our context, we are able to show:

\textbf{Theorem B} \textit{The horizontal Laplacian $\Delta_{\cD}$, as an unbounded operator on the Hilbert space $L^2(M,\mu)$, with domain $C^{\infty}(M)$, is essentially self-adjoint.}

\textbf{Theorem C} \textit{When $\cU(\cD)$ is a foliation, the horizontal Laplacian $\Delta_{\cD}$ is longitudinally hypoelliptic.}

Note that the notion of longitudinal hypoellipticity here is formulated using the scale of longitudinal Sobolev spaces $H^s(\cU(\cD))$ given in \cite{YK1}. Also note that the proof of Theorem C applies for the multiplier of $C^{\ast}(\cU(\cD))$ mentioned above. In order to prove the self-adjointness of this multiplier, it seems that one needs to generalise to multipliers as such the parametrix construction as it is done in \cite{YK2} in the case when $\cD$ is a constant rank distribution such that $\mathcal U(\cD)$ is a regular foliation. 
We leave this for future work.

Last, in the appendix we discuss some further developments. Specifically, in \S \ref{app:deRhamHodge}, we introduce the notion of smooth longitudinal differential forms for a generalised smooth distribution. Then we use our notions of local presentation and Riemannian metric to construct de Rham complex and a Hodge Laplacian of an arbitrary singular foliation. Finally, we introduce the notion of isometry for a Riemannian metric on a generalised smooth distribution (\S \ref{app:isometry}) and prove the invariance of the horizontal Laplacian under isometries.

\textbf{Notation:} Throughout the article $M$ is a smooth manifold with dimension $n$. We denote by $\cX(M)$ the $C^{\infty}(M)$-module of vector fields on $M$. Also, we denote by $\cX_c(M)$ the $C^{\infty}(M)$-module of compactly supported vector fields on $M$.

\textbf{Acknowledgements} I.A. would like to thank Konstantin Athanassopoulos, for several useful discussions on sub-Riemannian geometry. Y.K. is grateful to the Department of Mathematics of the National and Kapodistrian University of Athens and to the University Paris Diderot for hospitality and support. Both authors would like to thank Georges Skandalis for his suggestions.

\section{Smooth distributions}

We start in this section with our definition of a generalised smooth distribution, which includes the non-constant rank case, and give several examples. Then we introduce the notion of local presentations, the basic tool for our treatment of distributions as such.


\subsection{Distributions as modules of vector fields}

We start with the definition for distributions in terms of vector fields. We will need to recall the following from \cite[\S 1.1]{AS1}. 

\begin{enumerate}



\item Let $\cD$ be a $C^{\infty}(M)$-submodule of $\cX_c(M)$ and let $U$ be an open subset of $M$. Put $\iota_U : U \hookrightarrow M$ the inclusion map. For any vector field $X \in \cX(M)$ we write $X\left|_U\right. = X \circ \iota_U$. The \emph{restriction} of $\cD$ to $U$ is the $C^{\infty}(U)$-submodule of $\cX_c(U)$ generated by $f\cdot X\left|_U\right.$, where $f \in C^{\infty}_c(U)$ and $X \in \cD$. We denote this restriction $\cD\left|_U\right.$.

\item We say that the module $\cD$ is \emph{locally finitely generated} if, for every $x \in M$ there exist an open neighbourhood $U$ of $x$ and a finite number of vector fields $X_1,\ldots,X_k$ in $\cX(M)$ such that $\cD\left|_U\right. = C^{\infty}_c(U) \cdot X_1\left|_U\right. + \ldots + C^{\infty}_c(U) \cdot X_k\left|_U\right.$. We say that the vector fields $X_1,\ldots,X_k$ \emph{generate} the restriction $\cD\left|_U\right.$ of $\cD$ to $U$.
\end{enumerate}

We will also need the following construction, which is inspired from the notion of universal enveloping algebra. Recall that $(C^\infty(M), \cX_c(M))$ is a Lie-Rinehart algebra 
in the sense of \cite{Rinehart}. For the convenience of the reader, we recall this notion briefly: Let $R$ be a commutative ring with $1$. A Lie-Rinehart algebra \cite{Rinehart} is a pair $(A,L)$, where $A$ is a commutative $R$-algebra and $L$ a Lie algebra over $R$ which acts on $A$ by derivations and is also an $A$-module satisfying compatibility conditions that generalise the compatibility conditions between the structures of a $C^\infty(N)$-module and of a Lie algebra on the space $\cX(N)$ of smooth vector fields on a smooth manifold $N$.

\begin{enumerate}
\item[c)] Let $\cD$ be a $C^{\infty}(M)$-submodule of $\cX_c(M)$. The \emph{Lie-Rinehart subalgebra of $(C^\infty(M), \cX_c(M))$} associated to $\cD$ is the minimal submodule $\cU(\cD)$ of $\cX_c(M)$ which contains $\cD$ and is involutive, namely it satisfies $[X,Y] \in \cU(\cD)$ for every $X, Y \in \cU(\cD)$. Specifically, $\cU(\cD)$ is the $C^{\infty}(M)$-submodule of $\cX_{c}(M)$ generated by elements of $\cD$ and their iterated Lie brackets $[X_1,\ldots,[X_{k-1},X_k]]$ such that $X_i \in \cD$, $i = 1,\ldots,k$, for every $k \in \N$.
\end{enumerate}

We proceed now with our definition of smooth distribution, which focuses more on the dynamics involved. It is inspired by the definition of a singular foliation in \cite{AS1}.

\begin{definition}\label{dfn:distr}
A smooth distribution on $M$ is a locally finitely generated $C^{\infty}(M)$-submodule $\cD$ of the $C^{\infty}(M)$-module $\cX_c(M)$. We denote a distribution as a pair $(M,\cD)$.
\end{definition}

\begin{exs}\label{exs:distr}
\begin{enumerate}
\item A foliation $(M,\cF)$ in the sense of \cite{AS1} is a smooth distribution. Recall that $\cF$ is a locally finitely generated $C^{\infty}(M)$-submodule of $\cX_c(M)$ which is involutive, namely $[\cF,\cF]\subseteq \cF$. In particular, an arbitrary non-free action of a finite-dimensional Lie group on $M$ defines a foliation in the sense of \cite{AS1} and, therefore, a smooth distribution in the above sense.

\item Recall that an anchored vector bundle over $M$ is a vector bundle $E \to M$ endowed with a morphism of vector bundles $\rho : E \to TM$ over the identity diffeomorphism of $M$. The map $\rho$ induces a morphism of $C^{\infty}(M)$-modules $\Gamma_c E \to \cX_c(M)$, which we also denote $\rho$ by abuse of notation. Then the module $\cD_E = \rho(\Gamma_c E)$ is locally finitely generated: Indeed, if $\sigma_1,\ldots,\sigma_k$ is a frame of $E$ over an open $U \subset M$, the module $\cD_E\left|_U\right.$ is generated by the restrictions to $U$ of the vector fields $X_i = \rho(\sigma_i)$, $1 \leq i \leq k$. Whence $(M,\cD_E)$ is a smooth distribution.

\item When $\cU(\cD)$ is locally finitely generated, the pair $(M,\cU(\cD))$ is also a smooth distribution. In this case $(M,\cU(\cD))$ is a foliation in the sense of \cite{AS1}, since $\cU(\cD)$ is involutive by construction. Starting from a foliation $(M,\cF)$, the module $\cF$ is already involutive, whence $\cU(\cF)=\cF$.

\item Now start with a smooth distribution $(M,\cD)$ which is not a foliation. If $(M,\cU(\cD))$ is a foliation, then any other foliation $(M,\cF)$ such that $\cD \subseteq \cF$ contains $\cU(\cD)$. Whence $(M,\cU(\cD))$ is the smallest foliation which contains $(M,\cD)$.

\item 
Let $f \in C^{\infty}(\R^2)$ be defined by $f(x,y) = e^{-\frac{1}{x}}$ if $x >0$ and $f(x,y)=0$ if $x\leq 0$. Consider the smooth distribution $(\R^2,\cD)$ where $\cD$ is the $C^{\infty}_c(\R^2)$-module spanned by the vector fields $X = \partial_x$ and $Y=f\partial_y$. Note that $\cD$ is not involutive: Indeed, $[X,Y]=-x^{-2}X$ and the function $g(x,y)=x^{-2}$ is obviously not in $C^{\infty}(\R^2)$. We find that $\cU(\cD)$ is the module generated by $X$ and $Y_n$ where $Y_n(x,y) = x^{-n}f(x,y)\partial_y$ for all $n \in \N$. Whence $\cU(\cD)$ is not (locally) finitely generated.

\item Recall from \cite{Agrachev97}, \cite{Bellaiche} the general definition of a sub-Riemannian structure on the manifold $M$: This is an anchored vector bundle $\rho : E \to TM$ such that $\cD_E$ satisfies the bracket generating condition $\cU(\cD_E)=\cX_c(M)$. This definition covers both the equiregular and the non-equiregular sub-Riemannian structures. One finds important sub-Riemannian structures e.g. in $SU(2)$, the Heisenberg group, any contact manifold. In the next examples we recall some non-equiregular sub-Riemannian structures and their associated smooth distributions.

\item (Grushin plane.) Let $M=\R^2$ and $E=\R^2\times\R^2$. If $\sigma_1, \sigma_2$ is the standard frame of $E$, we define the map $\rho(\sigma_1) = \partial_x$, $\rho(\sigma_2)=x\partial_y$. That is, $\cD_{E}=\langle \partial_x, x\partial_y \rangle$. The $y$-axis is the set of singular points.

\item (Martinet space.) Let $M=\R^3$, $E$ the trivial bundle $\R^3 \times \R^2$ and $\rho$ the map which sends the standard frame of $E$ to the vector fields $X = \partial_x$, $Y = \partial_y + \frac{x^2}{2}\partial_z$. The $yz$-plane is the set of singular points.

\item Fix $f \in C^{\infty}(\R^4)$ and consider $\rho$ the map which sends the standard frame of the trivial bundle $\R^4\times \R^3$ to the vector fields $X = \partial_x$, $Y=\partial_y + x\partial_z + \frac{x^2}{2}\partial_w$, $Z=f\cdot\partial_w$. We find that the set of singular points is $S=f^{-1}(\{0\})$.
\end{enumerate} 
\end{exs}


To justify the terminology ``distribution'' in Definition \ref{dfn:distr} let us fix a smooth distribution $(M,\cD)$. Pick a point $x \in M$ and consider the $C^{\infty}(M)$-submodule $I_x\cD$, where $I_x = \{f \in C^{\infty}(M) : f(x)=0\}$. 
Since $\cD$ is locally finitely generated, the quotient $\cD_x = \cD/I_x\cD$ is a finite dimensional vector space. We call it the \textit{fiber} of $(M,\cD)$ at $x$. For any $X\in \cD$, we will denote by $[X]_x$ the corresponding class in $\cD_x$. We attach the following data to this vector space: 
\begin{enumerate}
\item There is a field of vector spaces $\cup_{x \in M}\cD_x$. The dimension map $\dim_{\cD} : M \to \N, x \mapsto \dim(\cD_x)$ is upper semicontinuous. 
\item Evaluation gives rise to a linear map $ev_x : \cD_x \to T_x M$. Put $D_x$ the image of this map. It is a vector subspace of $T_x M$. The field of vector spaces $\cup_{x \in M}D_x$ is a distribution of $M$ in the usual sense. The dimension map $\dim_D : M \to \N, x \mapsto \dim(D_x)$ is lower semicontinuous. 
\item Put $\cD(x) = \{X \in \cD : X(x)=0\}$ and $k^{\cD}_x = \cD(x)/I_x\cD$. The evaluation map $ev_x([X])=X(x)$ for every $X \in \cD$, gives rise to an exact sequence of vector spaces 
\begin{equation}\label{seq:distr}
0 \to k^{\cD}_x \to \cD_x \stackrel{ev_x}{\longrightarrow} D_x \to 0.
\end{equation}
\end{enumerate}

\begin{exs}\label{exs:fibcalc}
\begin{enumerate}
Let us look at the distribution $\cD_{E}$ arising from an anchored vector bundle $\rho : E \to TM$, as in item b) of examples \ref{exs:distr}. Fix a point $x$ in $M$. Recall from the Serre-Swan theorem that the fiber $E_x$ is the quotient of the $C^{\infty}(M)$-module $\Gamma_c E$ by the $C^{\infty}(M)$-submodule $I_x \Gamma_c E$ (\cf \cite{AZ5}). Since $\rho(I_x\Gamma_c E)\subseteq I_x \cD_E$ we obtain a linear epimorphism $\widehat{\rho_x} : E_x \to (\cD_E)_x$. Whence the dimension of the fiber $(\cD_E)_x$ at any $x \in M$ is bounded above by the rank of $E$.

\item Let us calculate explicitly the fibers of the distribution for the Grushin plane. First, if $p = (x,y)$ with $x \neq 0$, we have $\cD_{p}=\R^2 = D_p$ and $k_p^{\cD}=0$. To see this, consider $\lambda, \mu \in \R$ such that $\lambda[\partial_x]_p + \mu[x\partial_y]_p = 0$. This means that there exists $\phi \in C^{\infty}(\R^2)$ with $\phi(p)=0$ such that $\lambda\partial_x + \mu x \partial_y + \phi\partial_x + \phi x \partial_y = 0$. Evaluating this equation at $p$ we find $\lambda\partial_x(p) + \mu\partial_y(p) = 0$, whence $\lambda=\mu=0$.

Now take $p$ on the $y$-axis, namely $p=(0,y)$ for some $y \in \R$. We'll show that $\cD_p = \R^2$; since $D_p = \R$, this implies $k^{\cD}_p = \R$. To this end, we first show that $[x\partial_y]_p \in \cD_p$ does not vanish. Indeed, the vanishing of this element means that there exists $\phi \in C^{\infty}(\R^2)$ with $\phi(p)=0$ such that $x\partial_y = \phi(\alpha\partial_x + \beta x \partial_y)$ for some $\alpha, \beta \in C^{\infty}(\R^2)$. Whence $(1-\phi\beta)x\partial_y = \phi\alpha\partial_x$, which implies that $1-\phi\beta = 0$. Evaluating the latter at $p$ gives a contradiction.

Now take $\lambda, \mu \in \R$ such that $\lambda[\partial_x]_p + \mu [x\partial_y]_p = 0$. This means that there exist functions $\alpha, \beta \in C^{\infty}(\R^2)$ with $\alpha(p)=\beta(p)=0$, such that $\lambda\partial_x + \alpha\partial_x + \mu x \partial_y + \beta x \partial_y = 0$. Evaluating this at $p$ we find $\lambda=0$, therefore $\mu=0$ as well.

\item The fibers of both the Martinet space and the example in item i) can be calculated similarly with the Grushin plane. Notice that in all these three examples the dimension of $\cD_p$ is constant at every $p \in M$, whereas the dimension of $D_p$ is not constant. In fact, the field of vector spaces $\cup_{p \in M}\cD_p$ is nothing else than the trivial bundle $E$ mentioned in each of these examples.

\item A ``more singular'' example is item e) of examples \ref{exs:distr}. Let us calculate explicitly the exact sequence \eqref{seq:distr} for this example, at a point $p$ in $\R^2$. First, if $p = (x,y)$ with $x\lneq 0$ then the function $f$ vanishes in a neighbourhood $U$ of $p$, so $\cD\left|_U\right. = \langle \partial_x \rangle$. It is easy to see that $\cD_p = \R$ in this case; In fact, assuming $U$ is small enough so that it does not contain any points whose $x$-coordinate is $\geq 0$, we find that $\cD\left|_U\right.$ is a (regular) foliation whose leaves are lines parallel to the $x$-axis. Therefore $k_p^{\cD}=0$ and $\cD_p = D_p = \R$.

Second, if $p = (x,y)$ with $x \gneq 0$, there is a neighbourhood $U$ of $p$ such that the restriction of the function $f$ to $U$ is invertible. It follows that $\cD\left|_U\right. = \langle \partial_x, \partial_y \rangle$, whence $\cD_{p}=\R^2$; In this case, $\cD\left|_U\right.$ is just the foliation on $U$ by a single leaf. Therefore $k_p^{\cD}=0$ and $\cD_p = D_p = \R^2$.

The last case is when $p=(0,y)$ for some $y \in \R$. By the same calculation as in the case of the Grushin plane we find $\cD_p = \R^2$. The vector field $f\partial_y$ vanishes at $p=(0,y)$, whence $D_p=\R$ and the exact sequence \eqref{seq:distr} gives $k_p^{\cD}=\R$.
\end{enumerate}
\end{exs}

The fibers $\cD_{x}$ provide a way to find a minimal set of generators of $\cD$ locally. This is due to Prop. \ref{prop:distr} below, which is proven exactly\footnote{The proof given in \cite[Prop. 1.5]{AS1} does not make use of Lie brackets, so it holds for general distributions in the sense of our definition \ref{dfn:distr}.} as in \cite[Prop. 1.5]{AS1}.

\begin{prop}\label{prop:distr}
Let $(M,\cD)$ be a smooth distribution and $x \in M$. 
\begin{enumerate}
\item If $X_1,\ldots,X_k \in \cD$ are such that their images in $\cD_x$ give a basis of $\cD_x$, then there exists a neighborhood $U$ of $x$ such that $X_1, \ldots, X_k$ to $U$ generate $\cD\left|_U\right.$.
\item The dimension of $D_x$ is lower semicontinuous and the dimension of $\cD_x$ is upper semicontinuous.
\item The set of continuity of $x  \mapsto \dim(D_x)$ is $$\mathcal{C} = \{x \in M : ev_x : \cD_x \to D_x \text{ is bijective } \}$$ It is an open and dense subset of $M$. The restriction $\cD\left|_{\mathcal{C}}\right.$ is a projective $C^{\infty}(\mathcal{C})$-submodule of $\cX(\mathcal{C})$, whence it is the module of sections of a vector subbundle $D$ of $T\mathcal{C}$.
\end{enumerate}
\end{prop}

\begin{remark}
Note that in item c) of examples \ref{exs:fibcalc} the set of continuity is the complement of the $y$-axis in $\R^2$, so it has two connected components. In this case, the vector bundle $D$ mentioned in Prop. \ref{prop:distr} has rank $1$ on the component with negative $x$-coordinate and rank $2$ on the component with positive $x$-coordinate.
\end{remark}


\subsection{Local presentations}

Distributions which arise from anchored vector bundles are quite convenient; the anchored vector bundle plays the role of coordinates for the distribution. We localise this idea in the following definition.

\begin{definition}\label{dfn:presentation}
Let $(M,\cD)$ be a distribution and $U$ an open subset of $M$. 
\begin{enumerate}
\item A \emph{local presentation} of $(M,\cD)$ over $U$ is an anchored vector bundle $\rho_U : E_U \to TM$ (note that $E_U$ is a vector bundle over $U$), over the inclusion map $\iota_U : U \to M$, such that 
\[ 
\rho_U(\Gamma_c E_U) = \cD\left|_U\right..
\]
Once the distribution $(M,\cD)$ is fixed, a local presentation as such is denoted $(E_U,\rho_U)$. 
\item Let $W$ be an open subset of $U$. A morphism of local presentations from $(E_U,\rho_U)$ to $(E_W,\rho_W)$ is a vector bundles morphism $\psi : E_U \left|_W\right. \to E_W$ (over the identity) such that $\rho_W \circ \psi = \rho_U$. A morphism of local presentations from $(E_W,\rho_W)$ to $(E_U,\rho_U)$ is a vector bundles morphism $\phi : E_W \to E_U$ over the inclusion $\iota : W \hookrightarrow U$ such that $\rho_U \circ \phi = \rho_W$.
\item We say that a family of local presentations $\{(E_{U_i},\rho_{U_i})\}_{i \in I}$ covers $(M,\cD)$ if $\cup_{i \in I}U_i = M$.
\end{enumerate}
\end{definition}
Here are some immediate properties of a local presentation $(E_U,\rho_U)$:
\begin{enumerate}
\item When $U = M$ a \textit{presentation} of $(M,\cD)$ in terms of definition \ref{dfn:presentation} is  a vector bundle $E \to M$ together with a morphism of vector bundles $\rho : E \to TM$ over the identity. (Recall that sub-Riemannian manifolds come with a presentation as such by definition.)
\item Let $x \in U$. 
As in Examples~\ref{exs:fibcalc},
we get a linear epimorphism
\[
\widehat{\rho}_{U,x} : (E_U)_x \to \cD_x.
\]
Composing $\widehat{\rho}_{U,x}$ with the evaluation map we recover the restriction of $\rho_U$ to the fiber $(E_U)_x$. This is a linear epimorphism $\rho_{U,x} : (E_U)_x \to D_x$. Whence the following diagram commutes:
\begin{equation}\label{diag:trian}
\xymatrix{
(E_U)_x \ar@{->>}[r]^{\widehat{\rho}_{U,x}} \ar@{->>}[rd]_{\rho_{U,x}} & \cD_x \ar@{->>}[d]^{ev_x} \\
& D_x
}
\end{equation}
\end{enumerate}
Now let us show the existence of local presentations as such for any distribution.

\subsubsection{Minimal local presentations}\label{sec:minlocpre}

\begin{definition}\label{dfn:minchart}
Let $(M,\cD)$ be a distribution and $x$ a point of $M$. A local presentation $(E_U,\rho_U)$ of $(M,\cD)$ over a neighborhood $U$ of $x$ is called a \emph{minimal local presentation at $x$}, if the linear epimorphism $\widehat{\rho}_{U,x} : (E_U)_x \to \cD_x$ is an isomorphism.
\end{definition}

One can construct a minimal local presentation $(E_U,\rho_U)$ at $x \in M$ by the following recipe: 
\begin{itemize}
\item Consider the vector space $\cD_x$ and let $k \in \N$ be its dimension.
\item Choose a basis $\{[X_1]_x,\ldots,[X_k]_x\}$ of $\cD_x$. Also choose representatives $X_1,\ldots,X_k$ of the elements of this basis in $\cD$.
\item Take the neighborhood $U$ of $x$ to be the one for which it is proven in Proposition \ref{prop:distr} that $X_1\left|_U\right.,\ldots,X_k\left|_U\right.$ generate $\cD\left|_U\right.$.
\item Put $E_U$ the trivial bundle $U \times \R^k$.
\item Put $\rho_U : E_U \to TM$ the map $\rho_U(y,\lambda_1,\ldots,\lambda_k)=\lambda_1 X_1(y) + \ldots + \lambda_k X_k(y)$.
\item Obviously, at the level of sections we obtain a map $\rho_U : C^{\infty}_c(U)^k \to \cD\left|_U\right.$  defined by 
\[
\rho_U(f_1,\ldots,f_k)=f_1 \cdot X_1\left|_U\right. + \ldots + f_k \cdot X_k\left|_U\right.
\]
Whence $\rho_U(\Gamma_c E_U)=\cD\left|_U\right.$.
\end{itemize}

\begin{remarks}\label{rem:minchart}
\begin{enumerate}
\item Note that the local presentation $(E_U,\rho_U)$ we just constructed is not unique. It depends on the choice of basis for $\cD_x$, as well as the choice of representatives of elements of this basis. 
\item If we start from a point $x' \neq x$ the dimension of the bundle $E_{U'}$ might be different from the dimension of $E_U$ because in general $\dim(\cD_x) \neq \dim(\cD_{x'})$.
\item Of course one could just start with an arbitrary choice of generators for $\cD\left|_U\right.$ and construct a local presentation with the same recipe. But the dimension of the bundle $E_U$ we construct starting from a basis of $\cD_x$ is minimal. 
\item If we start with a different basis $\{[X'_1]_x,\ldots,[X'_k]_x\}$ of $\cD_x$, then, shrinking the neighborhood $U$ if necessary, the local presentation arising from the above construction will differ from $(E_U,\rho_U)$ only with respect to the anchor map. Namely, it will be the pair $(E_U,\rho'_U)$, where $E_U = U \times \R^k$ and $\rho'_U(y,\lambda_1,\ldots,\lambda_k) = \lambda_1 X'_1(y) + \ldots + \lambda_k X'_k(y)$ (\cf Proposition~\ref{propdef:equivmin1} below).
\end{enumerate}
\end{remarks}

\begin{ex}\label{ex:badminchart}
Let us give the minimal local presentations for item e) in examples \ref{exs:distr}. Let us start with a point $p$ on the $y$-axis of $\R^2$, for instance $p=(0,1)$. Since $\dim(\cD_p)=2$ there is an open neighbourhood $U_p$ of $p$ such that $\cD\left|_{U_p}\right.$ is generated by $\partial_x$ and $f\partial_y$. Put $E_{U_p}$ the trivial bundle $U_p \times \R^2$ and define $\rho_{U_p} : E_{U_p} \to T\R^2$ by $\rho_{U_p}(q,\lambda,\mu)=(q,\lambda\partial_x(q) + \mu f(q)\partial_y(q))$ for every $q \in U_p$ and $(\lambda,\mu) \in \R^2$.

Now let $p^{+}$ a point in $\R^2$ which lies to the right of the $y$-axis, namely its first coordinate is strictly positive. Since $\dim(\cD_{p^{+}})=2$, there is an open neighbourhood $U_{p^{+}}$ of $p^{+}$ such that $\cD\left|_{U_{p^{+}}}\right.$ is generated by $\partial_x$ and $\partial_y$. Put $E_{U_{p^{+}}}$ the trivial bundle $U_{p^{+}} \times \R^2$ and define $\rho_{U_{p^{+}}} : E_{U_{p^{+}}} \to T\R^2$ by $\rho_{U_{p^{+}}}(q,\lambda,\mu)=(q,\lambda\partial_x(q) + \mu \partial_y(q))$.

Now let $p^{-}$ a point in $\R^2$ which lies to the rleft of the $y$-axis, namely its first coordinate is strictly negative. Since $\dim(\cD_{p^{-}})=1$, there is an open neighbourhood $U_{p^{-}}$ of $p^{-}$ such that $\cD\left|_{U_{p^{-}}}\right.$ is generated by $\partial_x$. Put $E_{U_{p^{-}}}$ the trivial bundle $U_{p^{-}} \times \R$ and define $\rho_{U_{p^{-}}} : E_{U_{p^{-}}} \to T\R^2$ by $\rho_{U_{p^{-}}}(q,\lambda)=(q,\lambda\partial_x(q))$.
\end{ex}

\subsection{Equivalence of local presentations}\label{sec:equivpres}

Notice that in Example \ref{ex:badminchart}, the points $p^{-}, p^{+}$ may lie in the neighbourhood $U_p$. In this case the neighbourhoods $U_{p^{-}}$ and $U_{p^{+}}$ will have non-trivial intersections with $U_p$. This creates an ambiguity regarding the choice of minimal local presentation. Ambiguities as such are bound to arise in all cases, and not only for minimal local presentations. To deal with them we introduce a notion of equivalence for general local presentations.

\begin{definition}\label{dfn:equiv}
Let $(M,\cD)$ be a distribution and $U, V$ open subsets of $M$ such that $U \cap V \neq \emptyset$. Two local presentations $(E_U,\rho_U)$ and $(E_V,\rho_V)$ are called \textit{equivalent at a point $x \in U \cap V$}, if there exist an open neighbourhood  $W$ of $x$ such that $W \subset U \cap V$, a local presentation $(E_W,\rho_W)$ and morphisms of local presentations $\phi_{W,U} : (E_W,\rho_W) \to (E_U,\rho_U)$ and $\phi_{W,V} : (E_W,\rho_W) \to (E_V,\rho_V)$ such that $\rho_U\left|_W\right. \circ \phi_{W,U}=\rho_W = \rho_V\left|_W\right. \circ \phi_{W,V}$. 
\end{definition}
In other words, the following diagram commutes:
\begin{eqnarray}\label{eqn:equiv}
\xymatrix{
& E_U  \ar[rd]|-{\rho_{U}} & \\
E_W \ar[ru]|-{\phi_{W,U}} \ar[rd]|-{\phi_{W,V}} \ar[rr]|-{\rho_W} &  & TM \\
& E_V \ar[ru]|-{\rho_{V}} & \\
}
\end{eqnarray}
At the level of sections we have the following commutative diagram:
\begin{eqnarray}\label{eqn:equivmod}
\xymatrix{
& \Gamma E_U \ar[r]|-{\rho_{U}} & \cD\left|_U\right. \ar@{^{(}->}[rd]|-{\iota_{U,W}} & \\
\Gamma E_W \ar[ru]|-{\phi_{W,U}} \ar[rd]|-{\phi_{W,V}} \ar[rrr]|-{\rho_W} & & & \cD\left|_W\right. \\
& \Gamma E_V \ar[r]|-{\rho_{V}} & \cD\left|_V\right. \ar@{^{(}->}[ru]|-{\iota_{V,W}} &
}
\end{eqnarray}
where $\iota_{U,W} : \cD\left|_U\right. \to \cD\left|_W\right.$ is the restriction map $X\left|_U\right. \mapsto X\left|_W\right.$. It is easy to see that $\iota_{U,W} \circ \iota_{W,Z} = \iota_{U,Z}$ for appropriate open sets $U, W, Z$ of $M$.

\begin{lemma}
The relation introduced in Definition \ref{dfn:equiv} is an equivalence relation.
\end{lemma}
\begin{proof}
We just need to examine transitivity. Let $U, V, Z$ open subsets of $M$ such that $U \cap V \cap Z \neq \emptyset$ and $x \in U \cap V \cap Z$. Assume the local presentations $(E_U, \rho_U)$, $(E_V, \rho_V)$ are equivalent at the point $x$, and the same for $(E_V, \rho_V)$, $(E_Z, \rho_Z)$. Suppose these equivalences are realized by open neighborhoods $W$ of $x$ in $U \cap V$ and $W'$ of $x$ in $V \cap Z$, with respective local presentations $(E_W,\rho_W)$ and $(E_{W'},\rho_{W'})$.

Now consider the pullback vector bundle $E_{W\cap W'} := E_W \times_{(\phi_{W,V},\phi_{W',V})} E_{W'}$ over $W \cap W'$. Define $\rho_{W\cap W'} : E_{W\cap W'} \to TM$ by $\rho_{W\cap W'}(e,e') := \rho_V(\phi_{W,V}(e)) = \rho_V(\phi_{W',V}(e'))$. It is easy to check that $(E_{W\cap W'}, \rho_{W\cap W'})$, together with the maps $\phi_{W,U} \circ p_1 : E_{W\cap W'} \to E_U$ and $\phi_{W',Z} \circ p_2 : E_{W\cap W'} \to E_Z$ give an equivalence between the local presentations $(E_U, \rho_U)$ and $(E_Z, \rho_Z)$ at the point $x \in U \cap Z$.
\end{proof}

Moreover, it is easy to see that if the local presentations $(E_U, \rho_U)$, $(E_V, \rho_V)$ are equivalent at every point of $U \cap V$ and $(E_V, \rho_V)$, $(E_Z, \rho_Z)$ are equivalent at every point of $V \cap Z$ then $(E_U, \rho_U)$, $(E_Z, \rho_Z)$ are equivalent at every point of $U \cap V\cap Z$.



Now we prove that any two local presentations $(E_U,\rho_U)$ and $(E_V,\rho_V)$ with $U \cap V \neq \emptyset$ are equivalent at any point $x\in U \cap V$. For this purpose we use minimal local presentations. We will start with the following proposition.   

\begin{prop}\label{propdef:equivmin}
Let $x\in M$ and let $(E_U,\rho_U)$ be a local presentation defined in an open neighborhood $U$ of $x$. Then there exist a minimal local presentation $ (E_W,\rho_W)$ at $x$ defined in an open neighborhood $W\subset U$ of $x$ and a surjective morphism of local presentations $A_{U,W} : (E_U,\rho_U) \to (E_W,\rho_W)$. 
\end{prop}

\begin{proof}
Let $W\subset U$ be an open neighborhood of $x$ such that there exists a frame $\sigma_1,\ldots,\sigma_\ell$ of $E_U \left|_W\right.$ over $W$. So the restrictions of the vector fields $X_i = \rho(\sigma_i)$, $1 \leq i \leq \ell$, to $W$ generate the module $\cD\left|_W\right.$.

Let $k = \dim(\cD_x)$. There exist $Y_1,\ldots,Y_k \in \cD$ such that $Y_1\left|_W\right.,\ldots,Y_k\left|_W\right.$ generate $\cD\left|_W\right.$ (without loss of generality, we may assume with the same $W$!); put $(E_W = W \times \R^k,\rho_W)$ the associated minimal local presentation. 

Since $Y_1\left|_W\right.,\ldots,Y_k\left|_W\right.$ generate $\cD\left|_W\right.$, there exists a smooth map $A : W \to M_{\ell\times k}(\R)$ such that 
\[
\begin{bmatrix}
X_1(w) \\ 
\vdots \\
X_{\ell}(w)
\end{bmatrix}
= A(w) \cdot
\begin{bmatrix}
Y_1(w) \\
\vdots \\
Y_k(w)
\end{bmatrix}
\]
for every $w \in W$. Note that $A$ is not unique, since the $Y_i$'s  are merely generators of a module. This module may not be projective, whence they are not necessarily linearly independent.

Then, for every $w \in W$, we have
\[
\begin{bmatrix}
[X_1]_w \\ 
\vdots \\
[X_{\ell}]_w
\end{bmatrix}
= A(w) \cdot
\begin{bmatrix}
[Y_1]_w \\
\vdots \\
[Y_k]_w
\end{bmatrix},
\]
where $[X_i]_w$ and $[Y_j]_w$ are the classes of $X_i$ and $Y_j$ in $\cD_w$. Since the restrictions of $X_i$, $1 \leq i \leq \ell$, to $W$ generate the module $\cD\left|_W\right.$, the elements $[X_i]_w$, $1 \leq i \leq \ell$, generate the vector space $\cD_w$. Therefore, the rank of the matrix $A(w)$ is $\geq \dim \cD_w$. In particular, the rank of $A(x)$ is maximal and equals $k=\dim \cD_x$. Whence, shrinking $W$ if necessary, we can assume that the rank of $A(w)$ is maximal and equals $k$ for every $w\in W$. 

Observe that, since $[Y_1]_x, \ldots, [Y_k]_x$ is a basis in $\cD_x$, $A(x)$ is uniquely defined, that is, if $A, A' : W \to M_{\ell\times k}(\R)$ are as above, then we obtain 
\[
(A(x)-A'(x))
\begin{bmatrix}
[Y_1]_x \\
\vdots \\
[Y_k]_x
\end{bmatrix}
= 0,
\]
and, therefore, $A(x)=A'(x)$.

A matrix-valued map $A : W \to M_{\ell,k}(\R)$ introduced above gives rise to a surjective morphism of local presentations $A_{U,W} : (E_U,\rho_U) \to (E_W,\rho_W)$ defined by $A_{U,W}(w, \lambda) = (w,A(w)^t\cdot \lambda)$ for every $\lambda = (\lambda_1,\ldots,\lambda_\ell) \in \R^\ell$.
\end{proof}


\begin{ex}\label{ex:badA}
Let us apply the above to item e) in examples \ref{exs:distr}. With the notation of Ex. \ref{ex:badminchart}, choose $\{[\partial_x]_p, [f\partial_y]_p\}$ as a basis of $\cD_p = \R^2$, $\{[\partial_x]_{p^{-}}\}$ as a basis of $\cD_{p^{-}}=\R$ and $\{[\partial_x]_{p^{+}}, [f\partial_y]_{p^{+}}\}$ as a basis of $\cD_{p^{+}} = \R^2$. We can take $U_p$ to be $\R^2$, $U_{p^{-}}$ to be the half-plane to the left of the $y$-axis and $U_{p^{+}}$ to be the half-plane to the right of the $y$-axis. Then the map $A^{-} : U_{p^{-}} \to M_{1\times 2}(\R)$ is $A^{-}(w)=\left( \begin{smallmatrix} 1 & 0 \end{smallmatrix} \right)$ for all $w \in U_{p^{-}}$. The map $A^{+} : U_{p^{+}} \to M_{2 \times 2}(\R)$ is $A^{+}(w)=\left( \begin{smallmatrix} 1 & 0 \\  0 & 1 \end{smallmatrix} \right)$ for all $w \in U_{p^{+}}$.
\end{ex}

Slightly modifying the proof of Proposition~\ref{propdef:equivmin}, we get the following proposition. 

\begin{prop}\label{propdef:equivmin1}
Let $x\in M$ and $(E_U,\rho_U)$ and $(E_V,\rho_V)$ be minimal local presentations at $x$ defined in open neighborhoods $U$ and $V$ of $x$. Then there exist an open neighborhood $W\subset U\cap V$ of $x$ and an isomorphism of local presentations $(E_U \left|_W\right.,\rho_U) \cong (E_V \left|_W\right.,\rho_V)$. 
\end{prop}

\begin{prop}\label{prop:equivminlocpr}
Suppose that $U, V$ are open subsets of $M$ such that $U \cap V \neq \emptyset$. Then any local presentations $(E_U,\rho_U)$ and $(E_V,\rho_V)$ are equivalent at every $x \in U\cap V$.
\end{prop}

\begin{proof}
By Proposition \ref{propdef:equivmin}, there exist a minimal local presentation $ (E_{W_1},\rho_{W_1})$ at $x$ defined in an open neighborhood $W_1\subset U$ of $x$ and a surjective morphism of local presentations $A_{U,{W_1}} : (E_U,\rho_U) \to (E_{W_1},\rho_{W_1})$. Similarly, there exist a minimal local presentation $ (E_{W_2},\rho_{W_2})$ at $x$ defined in an open neighborhood $W_2\subset V$ of $x$ and a surjective morphism of local presentations $A_{V,{W_2}} : (E_V,\rho_V) \to (E_{W_2},\rho_{W_2})$. By Proposition~\ref{propdef:equivmin1}, we can assume that $W_1=W_2=W$ and $(E_{W_1},\rho_{W_1})=(E_{W_2},\rho_{W_2})=(E_{W},\rho_{W})$ is a minimal local presentation at $x$.

Put $E \to W$ the pullback vector bundle $E_U \left|_W\right. \times_{(A_{U,W},A_{V,W})} E_V \left|_W\right.$. Consider the map $\rho : E \to TM$ defined by $\rho(e_U,e_V)=\rho_U(e_U)=\rho_V(e_V)$. It is easy to see that $(E,\rho)$ is a local presentation of $(M,\cD)$ (over $W$), albeit not a minimal one. Put $\phi_U : E \to E_U \left|_W\right.$ and $\phi_V : E \to E_V \left|_W\right.$ the projection maps. We obtain a commutative diagram \eqref{eqn:equiv}.
\end{proof}


The results given in this section lead to a notion of atlas of local presentations for a smooth distribution. This will be discussed elsewhere.

\section{The Riemannian structure}

In this section we define the notion of Riemannian metric on a distribution $\cD$ and introduce a particular construction of a metric as such. This is necessary in order to associate a geometric Laplacian to a smooth distribution $\cD$ in \S \ref{sec:Lapl}. In Appendix \ref{app:isometry} we discuss isometries of distributions, using the notion of Riemannian metric we introduce here.

\subsection{Definition of Riemannian metric on a distribution}\label{sec:dfnRiem}

Here we will extend the classical definition of Riemannian structure on a vector bundle. So a Riemannian metric on a distribution $(M,\cD)$ needs to be defined on a family of pointwise linearizations of $\cD$, and must be smooth in some sense. The fibers $\cD_x = \cD / I_x \cD$ play the role of these linearizations, and the local presentations of $\cD$ can be used to make sense of this smoothness. But first we need the following, quite classical, facts:

\begin{enumerate}
\item Suppose that $(E, \langle \cdot, \cdot  \rangle_{E})$ and $(F, \langle \cdot, \cdot  \rangle_{F})$ are two (finite dimensional) Euclidean vector spaces with inner product and $A: E\to F$ is a linear epimorphism. Then we have the induced linear map $\bar A: E/\ker A \to F$, which is an isomorphism. 

The inner product $\langle \cdot, \cdot  \rangle_{E}$ induces an inner product $\langle \cdot, \cdot  \rangle_{E/\ker A}$ on $E/\ker A$, using the isomorphism $E/\ker A\cong (\ker A)^\bot$. 

We say that $A$ is a {\it Riemannian submersion}, if $\bar A$ preserves inner products:
\[
\langle \bar A u, \bar Av  \rangle_{F}= \langle u, v  \rangle_{E/\ker A}, \quad u,v \in E/\ker A. 
\] 

\item If $A: E\to F$ is a linear epimorphism and $\langle \cdot, \cdot  \rangle_{E}$ is an inner product on $E$, then there exists a unique inner product $\langle \cdot, \cdot  \rangle_F$ on $F$ such that $A : (E, \langle \cdot, \cdot  \rangle_{E})\to (F, \langle \cdot, \cdot  \rangle_{F})$ is a Riemannian submersion. This follows immediately from the fact that the induced map $\bar A: E/\ker A \to F$ is an isomorphism. The corresponding norm is given by 
\[
\|u\|_{F}= \|\bar A^{-1} u\|_{E/\ker A}=\inf \{\|w\|_E : w\in E, Aw=u\}, \quad u\in F. 
\] 
One sees easily that the norm $\|\cdot\|_F$ satisfies the parallelogram equality, whence it arises from an inner product $\langle \cdot,\cdot \rangle_F$.

\item If $(E, \langle \cdot, \cdot  \rangle_{E})$ and $(F, \langle \cdot, \cdot  \rangle_{F})$ are two Euclidean vector spaces and $A: E\to F$ is a linear epimorphism, then the adjoint $A^*: F \to E$ is a linear monomorphism. One can check that $A$ is a Riemannian submersion if and only if $A^*$ is an isometry, that is, preserves inner products:
\[
\langle A^* u, A^*v  \rangle_E= \langle u, v  \rangle_{F}, \quad \text{ for all } u,v \in F. 
\] 
\item Now let $(\cH,\langle \cdot, \cdot  \rangle_{\cH})$ be an infinite dimensional Hilbert space, $F$ a finite dimensional vector space and $A : \cH \to F$ a linear epimorphism. Since $A$ has finite rank, it is a compact map, whence for every $u \in F$ the infimum $\inf\{\|h\|_{\cH} : h\in \cH, Ah=u\}$ is attained at some $h \in \cH$. Put $\|u\|_{F}$ this infimum. Again, we find that $\|\cdot\|_F$ is a norm and it satisfies the rule of the paralellogram, whence it comes from an inner product $\langle \cdot,\cdot \rangle_F$. By construction, the map $A$ is a Riemannian submersion, that is, the induced map $\bar A: \mathcal H/\ker A \to F$ preserves the inner products.
\end{enumerate}

Now let us give the definition of a Riemannian metric. Its smoothness is formulated in terms of local presentations.




\begin{definition}\label{dfn:metric}
Let $(M,\cD)$ be a smooth distribution.
\begin{enumerate}
\item A \emph{Euclidean inner product} on $\cD$ is a family $\langle\ ,\ \rangle_{\cD}=\{\langle\cdot, \cdot\rangle_x, x\in M\}$ such that for every $x \in M$, $\langle\cdot, \cdot\rangle_x$ is a Euclidean inner product on $\cD_x$.
\item A \emph{local presentation of $\langle\ ,\ \rangle_{\cD}$ at $x \in M$} is a local presentation $\rho_U : E_U \to TM$ of $(M,\cD)$ over an open neighborhood $U$ of $x$ and a smooth family of inner products $\{\langle\cdot, \cdot\rangle_{(E_U)_y}, y\in U\}$ in the fibers of $E_U$ such that, for any $y\in U$, the linear epimorphism $(\rho_U)_y : (E_U)_y \to \cD_y$ is a Riemannian submersion. 
\item A \emph{Riemannian metric} on $(M,\cD)$ is a Euclidean inner product $\langle\ ,\ \rangle_{\cD}$ which is smooth in the following sense: For every $x \in M$ there exists a local presentation of $\langle\ ,\ \rangle_{\cD}$ at $x \in M$. 
\end{enumerate}
\end{definition}


\begin{remark}\label{rem:justifRiem}
Given a Riemannian metric $\langle\ ,\ \rangle_{\cD}=\{\langle\cdot, \cdot\rangle_x, x\in M\}$ on $(M,\cD)$, one can define the pointwise inner product of two elements $X,Y\in \cD$ as a function $\langle X, Y \rangle_{\cD}$ on $M$ given by
\[
\langle X, Y \rangle_{\cD}(x)=\langle [X]_x, [Y]_x \rangle_x,\quad x\in M.
\]
It should be noted that this function is, in general, non-smooth as one can see from the following example. This justifies the use of local presentations in definition \ref{dfn:metric} to express smoothness. In other words, it does not suffice to use the function $\langle\ ,\ \rangle_{\cD}$ for the definition of a Riemannian metric on the distribution $(M,\cD)$.

Consider the smooth distribution $(\R,\cD)$ where $\cD$ is the $C^{\infty}_c(\R)$-module spanned by the vector field $X=\varphi(x)\partial_x$ with some function $\varphi\in C^\infty_c(\RR)$ such that $\varphi(x)=0$ for $|x|\geq 1$ and $\varphi(x)>0$ for $|x|<1$. Note that $\cD$ is indeed involutive. Then $\cD_x=\RR$ for $|x|\leq 1$ and $\cD_x=0$ for $|x|>1$. Define a Euclidean inner product on $\cD$, setting $\langle [X]_x,  [X]_x \rangle_x=1$ for $|x|\leq 1$ and $\langle [X]_x,  [X]_x \rangle_x=0$ for $|x|>1$. One can check that it is smooth in the sense of definition \ref{dfn:metric} and, therefore, is a Riemannian metric on $(M,\cD)$. On the other hand, the function $\langle X, X\rangle_{\cD}$ is discontinuous at $x=\pm 1$.
\end{remark}

\begin{lemma}\label{lem:equivRiemmin}
For any $x\in M$, there exists a local presentation $(E_W,\rho_W)$ of the Riemannian metric on $\cD$ defined in an open neighborhood $W$ of $x$, which is minimal at $x$.
\end{lemma}

\begin{proof}
Let $(E_U,\rho_U)$ be a local presentation of the Riemannian metric defined in an open neighborhood $U$ of $x$. Then, by Proposition~\ref{propdef:equivmin}, there exist a minimal local presentation $(E_W,\rho_W)$ at $x$ defined in an open neighborhood $W\subset U$ of $x$ and a surjective morphism of local presentations $A_{U,W} : (E_U,\rho_U) \to (E_W,\rho_W)$. Using the recipe described in \S \ref{sec:dfnRiem} we obtain an inner product on $E_W$ so that, for any $y\in W$, $(\widehat{A}_{U,W})_y : (E_U)_y \to (E_W)_y$ is a Riemannian submersion. Since $\rho_W=\rho_U\left|_W\right. \circ A_{U,W}$, for any $y\in W$, the linear epimorphism $(\rho_W)_y : (E_W)_y \to \cD_y$ is a Riemannian submersion, and, therefore, $(E_W,\rho_W)$ is a local presentation of the Riemannian metric.
\end{proof}

\subsection{Construction of Riemannian metric}\label{sec:constRiem}

Here we prove Theorem A, namely the existence of Riemannian metrics for a distribution $(M,\cD)$ as in Dfn. \ref{dfn:metric}. Explicitely, we give a particular construction of a metric as such. This construction is not canonical, it depends on a certain choice; recall that the same happens with the familiar construction of a Riemannian metric for a smooth manifold. On the other hand, the (geometric) Laplacian we will construct in \S \ref{sec:Lapl} depends on the choice of Riemannian metric for $(M,\cD)$. Locally it is just a sum of squares. 

Since the module $\cD$ is locally finitely generated, there exists an at most countable, locally finite open cover $\{U_i\}_{i\in I}$ of $M$ and, for each $i\in I$, a finite number of vector fields $X^{(i)}_1,\ldots, X^{(i)}_{d_i}$ in $\cX_c(M)$, which generate the restriction $\cD\left|_{U_i}\right.$ of $\cD$ to $U_i$. Let $X_1, X_2, \ldots, X_N \in \cX_c(M)$ be the union of all families $X^{(i)}_1,\ldots, X^{(i)}_{d_i}$ over $i\in I$. (Here $N$ might be infinite.) It is easy to see that the vector fields $X_1, X_2, \ldots, X_N \in \cX_c(M)$ are global generators of the module $\cD$.

Now, if $N$ is finite, put $E^N$ the trivial Euclidean bundle $M \times \R^N$; if $N = \infty$ consider the Hilbert space $\ell^2=\left\{\{x_j\}_{j=1}^{\infty} : x_j\in \R, \sum_{j=0}^\infty x_j^2<\infty\right\}$ and put $E^{N}$ the trivial Hilbert bundle $M \times \ell^2$. We also consider the linear map $\rho^N : \Gamma(E^N) \to \cD$ defined by $\rho^N(f_1,\ldots,f_N)=f_1 X_1 + \ldots + f_N X_N$. Notice that, when $N=\infty$, this sum is finite at each point, since the cover $\{U_i\}_{i\in I}$ is locally finite. 
As above, we get a linear epimorphism
\[
\widehat{\rho}_x : (E^N)_x \to \cD_x
\]
for every $x\in M$. Using the recipe described in \S \ref{sec:dfnRiem} we obtain an inner product $\langle\cdot, \cdot\rangle_x$ on $\cD_x$ so that $\widehat{\rho}_{x}$ is a Riemannian submersion.

Now we have to check that the family $\{\langle\cdot, \cdot\rangle_x, x\in M\}$ of Euclidean inner products on $\cD_x$ is smooth. Fix $x\in M$. Even if the bundle $E^N$ is infinite-dimensional, it can be considered as a local presentation of $(M,\cD)$ over $M$, so one can apply Proposion~\ref{propdef:equivmin} to it. By this proposition, we will get that there exist a minimal local presentation $ (E_U,\rho_U)$ at $x$ and a surjective morphism of local presentations $Z_U : (E^N,\rho^N) \to (E_U,\rho_U)$. 
Using the recipe described in \S \ref{sec:dfnRiem} we obtain an inner product on $E_U$ so that, for any $y\in U$, $(\widehat{Z}_U)_y : E^N_y \to (E_U)_y$ is a Riemannian submersion.

It remains to show that, for any $y\in U$, the linear epimorphism $(\widehat{\rho}_U)_y : (E_U)_y \to \cD_y$ is a Riemannian submersion.   
Observe that the following diagram of linear maps commutes: 
\begin{eqnarray}\label{diag:Riemsubm1}
\xymatrix{
& (E_U)_y  \ar[dd]|-{(\widehat{\rho}_U)_y} & \\
E^N_y \ar[ru]|-{(\hat{Z}_U)_y} \ar[rd]|-{\widehat{\rho}_{y}} \\
& \cD_y
}
\end{eqnarray}

Recall that $(\widehat{\rho}_U)_y : (E_U)_y \to \cD_y$ is a Riemannian submersion if and only if $(\widehat{\rho}_U)^{\ast}_y : \cD_y \to (E_U)_y $ is an isometry. Since $\widehat{\rho}_y : E^N_y \to \cD_y$ and $(\widehat{Z}_U)_y : E^N_y \to (E_U)_y$ are Riemannian submersions, their adjoints $\widehat{\rho}^{\ast}_y : \cD_y\to E^N_y $ and $(\widehat{Z}_U)^{\ast}_y : (E_U)_y\to E^N_y$ are isometries. Using the commutativity of diagram \eqref{diag:Riemsubm1}, one can easily check that $(\widehat{\rho}_U)^{\ast}_y$ is an isometry as well, and, therefore, $(\widehat{\rho}_U)_y$ is a Riemannian submersion.

\subsection{Equivalence of local presentations of a Riemannian metric}\label{sec:equivRiem}

\begin{definition}
Let $U, V$ be open subsets of $M$ such that $U \cap V \neq \emptyset$ and $(E_U,\rho_U)$ and $(E_V,\rho_V)$ are local presentations of the Riemannian metric on $\cD$. We say that these local presentations are equivalent at a point $x \in U \cap V$, if there exist an open neighbourhood  $W$ of $x$ such that $W \subset U \cap V$, a local presentation $(E_W,\rho_W)$ of the Riemannian metric on $\cD$ and morphisms of local presentations $\phi_{W,U} : (E_W,\rho_W) \to (E_U,\rho_U)$ and $\phi_{W,V} : (E_W,\rho_W) \to (E_V,\rho_V)$, which are Riemannian submersions, such that $\rho_U\left|_W\right. \circ \phi_{W,U}=\rho_W = \rho_V\left|_W\right. \circ \phi_{W,V}$.
\end{definition}

\begin{lemma}\label{lem:equivRiem}
Let $U, V$ be open subsets of $M$ such that $U \cap V \neq \emptyset$. Any local presentations $(E_U,\rho_U)$ and $(E_V,\rho_V)$ of the Riemannian metric on $\cD$ are equivalent at any $x \in U \cap V$.
\end{lemma}

\begin{proof}
Let us first recall that, by Lemma~\ref{lem:equivRiemmin}, near an arbitrary point $x\in M$, the Riemannian metric on $\cD$ can be defined using a minimal local presentation $(E^0_W,\rho^0_W)$ at $x$; That is to say, $W$ is an open neighbourhood of $x$ in $U \cap V$ and $(E^0_W,\rho^0_W)$ is a local presentation of the Riemannian metric such that the rank of the vector bundle $E^0_W$ is equal to $\dim(\cD_x)$.

On the other hand, since $(E_U,\rho_U)$ and $(E_V,\rho_V)$ are local presentations of the Riemannian metric on $\cD$, for every $y \in U \cap V$, the maps $\widehat{\rho}_{U,y} \to \cD_y$ and $\widehat{\rho}_{V,y} \to \cD_y$ are Riemannian submersions.

Let us focus on $(E_U,\rho_U)$ for the moment: As shown in Proposition \ref{propdef:equivmin}, there exist an open neighborhood $\widetilde{W}$ of $x$ in $U\cap V$ and a minimal local presentation $(\widetilde{E}^0_{\widetilde{W}},\widetilde{\rho}^0_{\widetilde{W}})$ of $\cD$ at $x$, together with a morphism of local presentations $A_{U,\widetilde{W}} : E_U \to \widetilde{E}^0_{\widetilde{W}}$. Shrinking the neighborhoods $W$ and $\widetilde{W}$ if necessary and using proposition \ref{propdef:equivmin1}, we can assume that $(\widetilde{E}^0_{\widetilde{W}},\widetilde{\rho}^0_{\widetilde{W}})$ is the same as $(E^0_W,\rho^0_W)$. In view of this we denote the above morphism of local presentations by $A_{U,W} : E_U \to E^0_W$. At any $y \in W$ we have $\widehat{\rho}_{U,y} = \widehat{\rho}_{W,y} \circ A_{U,W,y}$. Since $\widehat{\rho}_{U,y}$ and $\widehat{\rho}_{W,y}$ are Riemannian submersions, an argument similar to the one at the end of \S \ref{sec:constRiem} shows that $A_{U,W,y}$ is also a Riemannian submersion. Likewise, starting from $(E_V,\rho_V)$ and shrinking the neighborhood $W$ if necessary, we find that the morphism of local presentations $A_{V,W} : E_V \to E^0_W$ is a Riemannian submersion.

Last, as in the proof of Prop. \ref{prop:equivminlocpr}, we put $E_W \to W$ the pullback vector bundle $E_U \left|_W\right. \times_{(A_{U,W},A_{V,W})} E_V \left|_W\right.$. Consider the map $\rho_W : E_W \to TM$ defined by $\rho(e_U,e_V)=\rho_U(e_U)=\rho_V(e_V)$. We obtain a commutative diagram \eqref{eqn:equiv}. The existence of the desired inner products in the fibers of $E_W$ follows from the linear algebra result proven in Lemma \ref{lem:Riemsubm2} below.
\end{proof}

\begin{lemma}\label{lem:Riemsubm2}
Let $A, B, X$ vector spaces with inner product and $\alpha : A \to X$, $\beta : B \to X$ Riemannian submersions. Then there exists an inner product on the pullback $C = A \times_{\alpha,\beta} B$ such that the projections $\pi_A : C \to A$ and $\pi_B : C \to B$ are Riemannian submersions.
\end{lemma}
\begin{proof}
The pullback $C$ is isomorphic to the vector space
\begin{eqnarray}\label{dirsum}
(\ker\alpha \times 0)\oplus(0 \times \ker\beta)\oplus\{(a,b) \in (\ker\alpha)^{\perp} \times (\ker\beta)^{\perp} : \alpha(a)=\beta(b)\}
\end{eqnarray}
Notice that the first term of the direct sum \eqref{dirsum} can be identified with a vector subspace of $A$, so it inherits the inner product of $A$. Put $\|\cdot\|_1$ for the induced norm. Likewise for the second term, which is a vector subspace of $B$; put $\|\cdot\|_2$ for the induced norm. The third term is isomorphic to $X$. In this term we consider the norm $\|(a,b)\|_3 = \|\alpha(a)\|=\|\beta(b)\|$. On the space $C$ we consider the norm $\|((a,0),(0,b),(a',b'))\| = \left( \|(a,0)\|_1^2 + \|(0,b)\|_2^2 + \|(a',b')\|_3^2 \right)^{1/2}$. It is easy to see that, since the norms $\|\cdot\|_i$, $i=1,2,3$ come from inner products, so does the norm $\|\cdot\|$.

Now write $A = \ker\alpha \oplus 0 \oplus (\ker\alpha)^{\perp}$ and notice that the restriction of $\pi_A$ to each term of \eqref{dirsum} is the first projection. In particular, the canonical inner products on $\ker\alpha$ and $\ker\beta$ make the restriction to the first term an isometry and the restriction to the second term an obvious Riemannian submersion. For the third term, since $\alpha$ is a Riemannian submersion we have $\|a\|=\|\alpha(a)\|$ and it follows that the projection is also a Riemannian submersion.
\end{proof}

\section{The horizontal differential of a distribution 
and its adjoint}\label{sec:Lapl}

\subsection{The dual of a distribution}\label{sec:dual}

Given a smooth distribution $(M,\cD)$, denote $\cD^*$ the disjoint union of vector spaces $\bigsqcup_{x\in M}\cD^*_x$. Recall that in \cite[Prop. 2.10]{AS2}, it was shown that $\cD^{\ast}$ is a locally compact space. Its topology (\cf \cite[\S 2.2]{AS2}) is the smallest topology which makes the following maps continuous:
\begin{itemize}
\item $p : \cD^{\ast} \to M$ is the projection $p(x,\xi)=x$.
\item For every $X \in \cD$ the map $q_X : \cD^{\ast} \to \R$ with $q_X(x,\xi)=\langle \xi, [X]_x \rangle$.
\end{itemize}

First, with the help of local presentations, we make sense of the smooth sections of this family of vector spaces. To this end, let us fix some notation first. Consider a local presentation $(E_U,\rho_U)$. Dualizing diagram \eqref{diag:trian}, for any $x\in U$, we obtain the commutative diagram: 

\begin{eqnarray}\label{diag:trian*}
\xymatrix{
 & \cD^{\ast}_x \ar[d]^{\widehat{\rho}^{\ast}_{U,x}} \\ 
T^{\ast}_x M \ar[r]_{\rho_{U,x}^{\ast}} \ar[ru]|-{ev^{\ast}_x} & E^{\ast}_{U,x} 
}
\end{eqnarray}
Note that, since $\widehat{\rho}_{U,x}$ is surjective, its dual map $\widehat{\rho}^{\ast}_{U,x}$ is injective.

\begin{definition}\label{dfn:smoothdual}
Let $\omega^{\ast}$ be a map $M \ni x \mapsto \omega^{\ast}(x) \in \cD^{\ast}_x$. We say that $\omega^{\ast}$ is a \emph{smooth} section of $\cD^{\ast}$ iff for every $x \in M$ there is a local presentation $(E_U,\rho_U)$ defined in a neighborhood of $x$ such that the section $\omega^{\ast}_U$ of the bundle $E_U^{\ast}$ defined by  $\omega^{\ast}_U(y)=\widehat{\rho}^{\ast}_{U,y} \circ \omega^{\ast}(y)$ for all $y \in U$ is smooth on $U$. We call $\omega^{\ast}_U$ a \emph{local realization} of $\omega^{\ast}$.
\end{definition}

\begin{notation}
We denote the set of smooth sections of $\cD^{\ast}$ by $C^{\infty}(M,\cD^{\ast})$ and its subset consisting of sections with compact support by $C^{\infty}_c(M,\cD^{\ast})$. Regarding the definition of the $C^{\infty}(M)$-module structure for $C^{\infty}_c(M,\cD^{\ast})$, it is $(f\cdot\omega^{\ast})(y) = f(y)\cdot \omega^{\ast}(y)$. Note that if $\omega^{\ast}_U$ is a local realization of $\omega^{\ast}$ then $f \left|_U\right. \cdot \omega^{\ast}_U$ is a local realization of $f\cdot \omega^{\ast}$.
\end{notation}


\begin{ex}\label{ex:D*}
Some elements of $C^{\infty}_c(M,\cD^{\ast})$ arise naturally from 1-forms on $M$ via the evaluation dual $ev^{\ast}$ in diagram \eqref{diag:trian*}. Namely we have a map $ev^{\ast} : \Omega^1_c(M) \to C^\infty_c(M,\cD^*)$: Every $\alpha \in \Omega^1_c(M)$ defines a map $M \ni x \mapsto ev^{\ast}(\alpha)(x) \in \cD_x^{\ast}$ by $ev^{\ast}(\alpha)(x)([X]_x) = \alpha_x(X)$ for every $X \in \cD$. Now, to show that $ev^{\ast}(\alpha)$ satisfies Definition \ref{dfn:smoothdual}, take an arbitrary local presentation $(E_U,\rho_U)$ and put $\alpha^{\ast}_{E_U} = \rho^{\ast}_U(\alpha)\in C^\infty(U,E^{\ast}_U)$. It follows from diagram \eqref{diag:trian*} that $(\widehat{\rho}_{U,y} \circ ev^{\ast}(\alpha))(y) = \alpha^{\ast}_{E_U}(y)$ for all $y \in U$. So $\alpha^{\ast}_{E_U}$ is the local realization of $ev^{\ast}(\alpha)$. 
Notice that $\alpha^{\ast}_{E_U}$ vanishes on the kernel of $\rho_U$. This is rather remarkable, given that the dimension of $\ker\rho_{U,y}$ is not constant as we change the point $y$ in $U$.
\end{ex}



We used local presentations in order to define the $C^{\infty}(M)$-module $C^{\infty}_c(M,\cD^{\ast})$. This module also admits a description which does not use local presentations, as explained in Proposition \ref{not:smoothdual} below.

\begin{prop}\label{not:smoothdual}
Let $\omega^{\ast}$ be a map $M \ni x \mapsto \omega^{\ast}(x) \in \cD^{\ast}_x$. If $\omega^*\in C^\infty(M,\cD^*)$, then the function $M \ni x \mapsto \langle \omega^{\ast}(x), [X]_x \rangle$ is smooth on $M$ for any $X\in\cD$. Conversely, if the function $M \ni x \mapsto \langle \omega^{\ast}(x), [X]_x \rangle$ is smooth on $M$ for any $X\in\cD$ and $(E_V,\rho_V)$ is an arbitrary local presentation of $\cD$, then the local realization $\omega^{\ast}_V$ of $\omega^{\ast}$ is smooth on $V$. 
\end{prop}

\begin{proof}
Let $\omega^*\in C^\infty(M,\cD^*)$. Then for every $x \in M$ there is a local presentation $(E_U,\rho_U)$ defined in a neighborhood of $x$ such that the local realization $\omega^{\ast}_U$ is smooth on $U$. We may assume that there exists a local frame $\sigma_1,\ldots,\sigma_d$ of $E_U$ defined on $U$. Consider $X\in \cD$, supported in $U$. We can write 
\[
X=\sum_{i=1}^d a_i\rho_U(\sigma_i)
\] 
with some $a_i\in C^\infty_c(U)$. Then we have 
\[
\langle \omega^{\ast}(x), [X]_x \rangle= \sum_{i=1}^d a_i(x)\langle \rho^*_{U,x}\omega^{\ast}(x), \sigma_i(x)\rangle= \sum_{i=1}^d a_i(x)\langle \omega^{\ast}_U(x), \sigma_i(x)\rangle,
\]
which depends smoothly on $x\in U$. For the proof in the case of an arbitrary $X\in \cD$, we use a appropriate covering of $M$ and a subordinated partition of unity.

On the other hand, assume that the function $M \ni x \mapsto \langle \omega^{\ast}(x), [X]_x \rangle$ is smooth on $M$ for any $X\in\cD$. Let $(E_V,\rho_V)$ be an arbitrary local presentation of $\cD$. For any $x\in V$, let $U\subset V$ be an open neighborhood of $x$ such that that there exists a local frame $\sigma_1,\ldots,\sigma_d\in C^\infty(U,E_V\left|_U\right.)$ of $E_V\left|_U\right.$. Let  $\sigma^\#_1,\ldots,\sigma^\#_d\in C^\infty(U,E^*_V\left|_U\right.)$ be the dual local frame of $E^*_V\left|_U\right.$. Then, for any $y\in U$, we can write 
\[
\omega^{\ast}_V(y)=\sum_{i=1}^d\langle \omega^{\ast}_V(y), \sigma_i(y)\rangle \sigma^\#_i(y)=\sum_{i=1}^d\langle \omega^{\ast}(y), \rho_{U,y}\sigma_i(y)\rangle \sigma^\#_i(y)=\sum_{i=1}^d\langle \omega^{\ast}(y), [\rho_{U}(\sigma_i)]_y)\rangle \sigma^\#_i(y),
\]
that proves smoothness of $\omega^{\ast}_V$ on $U$.
\end{proof}



\begin{cor}\label{cor:smoothdual}
There exists a bilinear pairing 
$$
C^{\infty}_c(M,\cD^{\ast}) \otimes_{C^{\infty}_c(M)} \cD \to C^{\infty}_c(M).
$$
\end{cor}

One can easily check that this pairing is non-degenerate. Whence the $C^{\infty}(M)$-module $C^{\infty}_c(M,\cD^{\ast})$ is in duality with the $C^{\infty}(M)$-module $\cD$.



\subsubsection{The Riemannian metric of the dual}\label{sec:Riemdual}



Given a Riemannian metric $\langle\ ,\ \rangle_{\cD}=\{\langle\cdot, \cdot\rangle_x, x\in M\}$ on $(M,\cD)$, one can define a family $\langle\ ,\ \rangle_{\cD^*}=\{\langle\cdot, \cdot\rangle_{x}, x\in M\}$ of inner products on $\cD^*_x$ and the pointwise inner product of two elements $\omega,\omega^\prime \in C^\infty(M,\cD^*)$ as a function $\langle \omega, \omega^\prime \rangle_{\cD^*}$ on $M$ given by
\[
\langle \omega, \omega^\prime \rangle_{\cD^*}(x)=\langle \omega(x), \omega^\prime(x) \rangle_x,\quad x\in M.
\]

Unlike the case of $\cD$ (\cf Remark~\ref{rem:justifRiem}), one can prove the following regularity property of the pointwise inner product on $\cD^*$.

\begin{lemma}\label{lem:D*smooth}
For any $\omega,\omega^\prime\in C^\infty(M,\cD^*)$, we have $\langle \omega, \omega^\prime \rangle_{\cD^*}\in C^\infty(M)$. 
\end{lemma}
\begin{proof}
Take an arbitrary local presentation $(E_U,\rho_U)$ defined in an open subset $U\subset M$. Then the local realizations $\omega_U$ and $\omega^\prime_U$ of $\omega$ and $\omega^\prime$ respectively are smooth on $U$. Since $\widehat{\rho}^{\ast}_{U,x} : \cD^*_x\to E^\ast_{U,x}$ is an isometry for any $x\in U$, we have 
\[
\langle \omega, \omega^\prime \rangle_{\cD^*}(x)=\langle \omega_U(x), \omega^\prime_U(x) \rangle_{E^\ast_{U,x}},\quad x\in U,
\]
that immediately implies that $\langle \omega, \omega^\prime \rangle_{\cD^*}$ is smooth on $U$.
\end{proof}

\subsection{The horizontal differential and its adjoint}\label{sec:hordif}

In view of the above, we are now ready to give the definition of the horizontal differential of a distribution.

\begin{definition}\label{dfn:hordif}
Let $(M,\cD)$ be a smooth distribution.  
\begin{enumerate}
\item The  \textit{horizontal differential} is the operator 
$d_{\cD} : C^{\infty}_c(M) \to C^{\infty}_c(M,\cD^{\ast})$ defined as $d_{\cD} = ev^{\ast}\circ d$, where $d : C^{\infty}_c(M) \to \Omega^1_c(M)$ is the de Rham differential.
\item Given a local presentation $(E_U,\rho_U)$, put $d_{E^{\ast}_U} : C^{\infty}_c(U) \to C^{\infty}_c(U,E^{\ast}_U)$ the operator defined as the composition of the de Rham differential $d : C^{\infty}_c(U) \to \Omega^1_c(U)$ with the map $\rho_U^{\ast} : \Omega^1_c(U) \to C^{\infty}_c(U,E_U^{\ast})$. We call $d_{E^{\ast}_U}$ a \emph{local presentation} of the horizontal differential $d_{\cD}$.
\end{enumerate}
\end{definition}

Note that the terminology ``local presentation'' for the operator $d_{E^{\ast}_U}$ is justified by the following commutative diagram:
\begin{eqnarray}\label{diag:hordifloc}
\xymatrix{
C^{\infty}_c(U) \ar[r]^{d} & \Omega^1_c(U)  \ar[rd]_{ev^{\ast}} \ar[rr]^{\rho^{\ast}_U} & & C^{\infty}_c(U,E^{\ast}_U) \\
 & & C^{\infty}_c(U,\cD^{\ast}) \ar[ru]_{\widehat{\rho}^{\ast}_U} &
}
\end{eqnarray}
Thus, we have
\[
d_{E^{\ast}_U}=\widehat{\rho}^{\ast}_U\circ d_{\cD}. 
\]

Now let us fix a Riemannian metric on the distribution $(M,\cD)$, as in Definition \ref{dfn:metric}, and a positive smooth density $\mu$ on $M$.

A naive approach to introducing an adjoint for the operator $d_{\cD}=ev^{\ast}\circ d$ would be to use a Riemannian metric on $M$ in order to make sense of the adjoint of the usual de Rham differential $d^{\ast}$. But such a metric would have to be somehow compatible with the Riemannian metric of the distribution $(M,\cD)$, and this reduces considerably the range of applicability of our constructions.

Instead, we will show in this section that an adjoint can be constructed only with the data of the Riemannian metric on the distribution and the smooth density of $M$, for which no compatibility is required. This is possible thanks to the local presentations of our Riemannian metric.

Whence, with the above data we have:


\begin{enumerate}

\item Given a local presentation $(E_U,\rho_U)$ of the  Riemannian metric on $(M,\cD)$, first we can define an inner product on $C^{\infty}_c(U,E_U^{\ast})$ by $$(\omega^{\ast}_{U,1},\omega^{\ast}_{U,2})_{L^2(U,E_U,\mu)} = \int_U \langle \omega^{\ast}_{U,1}(y),\omega^{\ast}_{U,2}(y) \rangle_{E^\ast_{U,x}}d\mu(y)$$ (We denote $L^2(U,E^{\ast}_U,\mu)$ the completion of $C^{\infty}_c(U,E_U^{\ast})$ with respect to the norm$\|\cdot\|_{L^2(U,E_U^{\ast},\mu)}$ associated to this inner product.) 
Since $d_{E^{\ast}_U}$ is a first order differential operator, we can define its adjoint $d_{E^{\ast}_U}^{\ast} : C^{\infty}_c(U,E^{\ast}_U) \to C^{\infty}_c(U)$ by 
\[
(d_{E^{\ast}_U}^{\ast}\omega^{\ast}_U,\alpha )_{L^2(U,\mu)} = (\omega^{\ast}_U,d_{E^{\ast}_U}\alpha)_{L^2(U,E^{\ast}_U,\mu)} \text{ for all } \omega^{\ast}_U \in C^{\infty}_c(U,E^{\ast}_U) \text{ and } \alpha \in C^{\infty}_c(U).
\]

\item We can also define an inner product on $C^\infty_c(M,\cD^*)$ by
\[
(\omega,\omega^\prime)_{L^2(M,\cD^*,\mu)}=\int_M \langle \omega, \omega^\prime \rangle_{\cD^*}(x) d\mu(x),\quad \omega,\omega^\prime\in C^\infty(M,\cD^*).
\]

By Lemma~\ref{lem:D*smooth}, the function $\langle \omega, \omega^\prime \rangle_{\cD^*}$ is smooth, so the integral is well-defined. We denote $L^2(M,\cD^*,\mu)$ the completion of $C^\infty_c(M,\cD^*)$ with respect to the norm $\|\cdot\|_{L^2(M,\cD^*,\mu)}$.
\end{enumerate}

Since $\cD^*$ is not a vector bundle, the existence of the adjoint $d_{\cD}^{\ast} : C^{\infty}_c(M,\cD^{\ast}) \to C^{\infty}_c(M)$ of the operator $d_{\cD} : C^{\infty}_c(M) \to C^{\infty}_c(M,\cD^{\ast})$ is not immediate. We will show that such an adjoint arises from the adjoints $d_{E^{\ast}_U}^{\ast}$ of the local presentations $d_{E^{\ast}_U}$.

To make a start with explaining this, let us first fix a local presentation $(E_U,\rho_U)$. Now take $\omega^{\ast} \in C^{\infty}_c(U,\cD^{\ast})$. Let $\omega^{\ast}_U \in C^{\infty}_c(U,E^{\ast}_U)$ be the local realization of $\omega^{\ast}$: $ \omega^{\ast}_U(y)=\widehat{\rho}^{\ast}_{U,y} \circ \omega^{\ast}(y)$ for all $y \in U$. Define $d^{\ast}_{\cD,U} \omega^{\ast}\in C^\infty_c(U)$ by
\begin{equation}\label{eqn:evadj}
d^{\ast}_{\cD,U} \omega^{\ast} (y) = d_{E^{\ast}_U}^{\ast} \omega^{\ast}_U(y) \text{ for all } y \in U.
\end{equation}
\begin{lemma}\label{lem:adjhordif1}
The operator $d^{\ast}_{\cD,U} : C^{\infty}_c(U,\cD^{\ast}) \to C^{\infty}_c(U)$ is adjoint to $d_{\cD}\left|_{U}\right.$.
\end{lemma}
\begin{proof}
For any $\omega^{\ast} \in C^{\infty}_c(U,\cD^{\ast})$ and $\alpha \in C^{\infty}_c(U)$, we have:
\begin{multline*}
\left( d^{\ast}_{\cD,U} \omega^{\ast}, \alpha \right)_{L^2(U,\mu)} = \int_U {d_{\cD}^{\ast} \omega^{\ast}(y)}\alpha(y) d\mu(y) = \int_U {d_{E^{\ast}_U}^{\ast} \omega^{\ast}_U(y)}\alpha(y)  d\mu(y) \\ = \int_{U} \langle \omega^{\ast}_U(y), d_{E^{\ast}_U} \alpha(y) \rangle_{E^*_{U,y}} d\mu(y)= \int_{U} \langle \widehat{\rho}^{\ast}_{U,y} \circ \omega^{\ast}(y), \widehat{\rho}^{\ast}_{U,y}\circ d_{\cD} \alpha(y) \rangle_{E^*_{U,y}} d\mu(y)\\ = \int_{U} \langle \omega^{\ast}(y), d_{\cD} \alpha(y) \rangle_{\cD^*_y} d\mu(y)=\left( \omega^{\ast}, d_{\cD} \alpha \right)_{L^2(U,\cD^*,\mu)},
\end{multline*}
where we used the commutative triangle in diagram \eqref{diag:hordifloc} and the fact that $\widehat{\rho}_{U,y}^{\ast}: \cD_y^* \to E^*_{U,y}$ is an isometry.
\end{proof}

In order to show that $d^{\ast}_{\cD,U}$ can be extended to an adjoint $d^{\ast}_{\cD}$ of $d_{\cD}$ (over the whole of $M$ instead of just $U$), we need to prove that $d^{\ast}_{\cD,U}$ does not depend on the choice of local presentation $(E_U,\rho_U)$. For this, we have to show that,
given open subsets $U, V$ of $M$ such that $U \cap V \neq \emptyset$ and  local presentations $(E_U,\rho_U)$ and $(E_V,\rho_V)$ of the Riemannian metric on $\cD$, for any $\omega^{\ast} \in C^{\infty}_c(U\cap V,\cD^{\ast})$, we have
\[
d^{\ast}_{\cD,U}\omega^{\ast}=d^{\ast}_{\cD,V}\omega^{\ast} \in C^{\infty}_c(U\cap V).
\]
This immediately follows from Lemma~\ref{lem:adjhordif1}, because, for any $\alpha \in C^{\infty}_c(U\cap V)$, we have
\[
\left( d^{\ast}_{\cD,U} \omega^{\ast}-d^{\ast}_{\cD,V} \omega^{\ast}, \alpha \right)_{L^2(M,\mu)}=\left( \omega^{\ast}, d_{\cD} \alpha \right)_{L^2(U,\cD^*,\mu)}-\left( \omega^{\ast}, d_{\cD} \alpha \right)_{L^2(V,\cD^*,\mu)}=0.
\]

So we just proved the following result:

\begin{prop}\label{prop:adjhordif}
There exists a unique operator $d^{\ast}_{\cD} : C^{\infty}_c(M,\cD^{\ast}) \to C^{\infty}_c(M)$ which is adjoint of the horizontal differential $d_{\cD}$. The local presentation of $d^{\ast}_{\cD,U}$ is $d^{\ast}_{E^{\ast}_U}$.
\end{prop}

\section{The horizontal Laplacian of a distribution}\label{sec:horLapl}

\subsection{The definition}

Now we are able to define the horizontal Laplacian of a distribution.

\begin{definition}\label{dfn:horlapl}
Let $(M,\cD)$ be a smooth distribution. Choose a Riemannian metric on $\cD$ and a positive smooth density $\mu$ on $M$. 
\begin{enumerate}
\item The operator $\Delta_{\cD} = d^{\ast}_{\cD} \circ d_{\cD} : C^{\infty}_c(M) \to C^{\infty}_c(M)$ is called the \emph{horizontal Laplacian} of the distribution $(M,\cD)$.
\item Given an open subset $U\subset M$, the operator $\Delta_{\cD,U}=d^{\ast}_{\cD,U}\circ (d_{\cD}\left|_{U}\right.) : C^{\infty}_c(U) \to C^{\infty}_c(U)$ is called the \emph{restriction} of $\Delta_{\cD}$ to $U$.
\item Given a local presentation  $(E_U,\rho_U)$ of $(M,\cD)$. the operator $\Delta_{E_U} = d^{\ast}_{E_U^{\ast}}\circ d_{E_U^{\ast}} : C^{\infty}_c(U) \to C^{\infty}_c(U)$ is called a \emph{local presentation} of the horizontal Laplacian $\Delta_{\cD}$.
\end{enumerate}
\end{definition}

\begin{remarks}\label{rems:horlapl}
\begin{enumerate}
\item Definition \ref{dfn:horlapl} is quite geometric, as it uses the Riemannian metric of the distribution (and a positive density on $M$). Notice that $\Delta_{\cD,U}=\Delta_{E_U}$. This shows that, locally, the horizontal Laplacian $\Delta_{\cD}$ is nicely controlled by its local presentations $(E_U,\rho_U)$. In appendix \ref{app:isometry} we discuss the relation between horizontal Laplacians via an isometry (\cf proposition \ref{prop:isometry}).

\item Also $\Delta_{\cD,U}$ can be described using the quadratic form associated with the inner product of $E_U$: 
\begin{equation}\label{eqn:quadrloc}
(\Delta_{\cD,U}u,u)=\int_U \|d_{E_{U}}u(x)\|_{E^*_{U,x}}^{2}d\mu(x) \text{ for all } u \in C^{\infty}_c(U).
\end{equation}
Actually this integral formula holds globally, using the inner product of the fibers $\cD_x$: 
\begin{equation}\label{eqn:quadrglob}
(\Delta_{\cD}u,u)=\int_M \|d_{\cD}u(x)\|_{\cD^*_x}^{2}d\mu(x) \text{ for all } u \in C^{\infty}_c(M).
\end{equation}

\item Locally, the horizontal Laplacian also admits a ``sum of squares'' description:
Choose an orthonormal frame $(\omega_1,\ldots,\omega_d)$ of $E_U$. Then $\rho_{U}(\omega_1), \ldots, \rho_{U}(\omega_d) \in \cD\left|_{U}\right.$ generate $\cD\left|_{U}\right.$ and we have
\begin{equation}\label{e:sum_of_sq}
\Delta_{\cD,U} = \sum_{i=1}^d \rho_{U}(\omega_i)^{\ast}\rho_{U}(\omega_i).
\end{equation} 
To see this, we can use formula \eqref{eqn:quadrloc}. Denote by $(\omega^\#_1,\ldots, \omega^\#_d)$ the dual orthonormal frame of $E^*_U$. Then, for any $u\in C^\infty_c(U)$, we have 
\begin{multline*}
\|d_{E_{U}}u(x)\|_{E^*_{U,x}}^{2} = \sum_{i=1}^d \left|\langle d_{E_{U}}u(x), \omega^\#_i(x)\rangle _{E^*_{U,x}}\right|^2=\sum_{i=1}^d \left|\langle d_{E_{U}}u(x), \omega_i(x)\rangle\right|^2 \\ = \sum_{i=1}^d \left|\langle du(x), \rho_U[\omega_i(x)]\rangle\right|^2=\sum_{i=1}^d \left|\rho_U[\omega_i]u(x)\right|^2.    
\end{multline*}
Here, by the same notation $\langle\cdot,\cdot\rangle$, we denote the duality between $E^*_U$ and $E_U$ and the duality between $T^*U$ and $TU$. By \eqref{eqn:quadrloc}, we get 
\[
(\Delta_{\cD,U}u,u)=\sum_{i=1}^d \|\rho_U[\omega_i]u\|_{L^2}^2=\sum_{i=1}^d (\rho_U[\omega_i]^*\rho_U[\omega_i]u,u),    
\]
that implies \eqref{e:sum_of_sq}.
\item We will consider $\Delta_{\cD}$ as an unbounded linear operator on the Hilbert space $L^2(M,\mu)$ with domain $C^{\infty}_c(M)$.
\end{enumerate}
\end{remarks}

\subsection{Symbol of the horizontal Laplacian}\label{sec:sumbhor}

The notion of the principal symbol of a operator is  connected with some algebra of differential or pseudodifferential operators. Usually, it is a homomorphism from this algebra to an algebra of symbols. Whence, in order to speak about the principal symbol of the horizontal Laplacian $\Delta_{\cD}$, we need to ensure that it belongs in some pseudodifferential calculus.

Since $\Delta_{\cD}$ is a second order differential operator on $M$, the obvious choice of pseudodifferential calculus for it is the standard calculus of the manifold $M$. From this viewpoint, its principal symbol $\sigma_{\Delta_{\cD}}$ is a smooth function on $T^*M$, homogeneous of degree $2$. Recall that any vector field $X$ on $M$ is a first order differential operator on $M$, whose principal symbol is given by 
\[
\sigma_X(x,\xi)=\langle X(x), \xi\rangle, \quad x\in M, \xi\in T^*_xM. 
\] 
Using properties of the principal symbol and the ``sum of squares'' description \eqref{e:sum_of_sq}, we get
\begin{multline}\label{e:sigmaD}
\sigma_{\Delta_{\cD,U}}(x,\xi)=\sum_{i=1}^d \left|\langle \rho_U[\omega_i](x), \xi\rangle \right|^2 =\sum_{i=1}^d \left|\langle \omega_i, \rho^*_{U,x}(\xi)\rangle \right|^2=\left| \rho^*_{U,x}(\xi) \right|^2_{E^*_{U,x}}=\left|ev^*_{x}(\xi) \right|^2_{\cD^*_{x}}, 
\end{multline}
for all $x \in U$ and $\xi \in T^{\ast}_x M$. Here, at the last step, we used the diagram \eqref{diag:trian*} and the fact that $\rho_{U,x}^{\ast}$ is an isometry.

\begin{remark}\label{rem:symbhor}
The equality \eqref{e:sigmaD} 
suggests that there should be a construction of the principal symbol  of $\Delta_{\cD}$ as a function on the locally compact space $\cD^{\ast}$. Such a symbol would carry information about the module $\cD$ rather than the manifold $M$. First, notice that every $X \in \cD$ gives rise to a symbol $\sigma_{X} : \cD_x^{\ast} \to \C$ 
given by $\sigma_X(x,\xi)=\langle [X]_x,\xi \rangle $ for any $x\in M$ and $\xi\in \mathcal D_x$. Then, as in \eqref{e:sigmaD},
it makes sense to define $\sigma_{\Delta_{\cD},U} = \sum_{i=1}^d \overline{\sigma_{\rho_U[\omega_i]}}\sigma_{\rho_U[\omega_i]}$, in other words $\sigma_{\Delta_{\cD},U}(x,\xi) = \left|\xi\right|^2_{\cD^{\ast}_x}$ for every $(x,\xi) \in \cD_x^{\ast}$. However, since the module $\cD$ is \emph{not} necessarily involutive, one cannot associate a pseudodifferential calculus to the distribution $(M,\cD)$. Indeed, it is easy to see that the algebra of differential operators on $M$ generated by $\cD$ coincides with the algebra of differential operators on $M$ generated by the minimal Lie-Rinehart algebra $\cU(\cD)$ of the distribution $\cD$. From this point of view, the symbol we just constructed is meaningless. However this discussion gives rise to a second viewpoint on the horizontal Laplacian $\Delta_{\cD}$ and its principal symbol, which we explain in \S \ref{sec:symblonghor} below.
\end{remark}

\subsubsection{The longitudinal symbol}\label{sec:symblonghor}

Now put $\cF$ the minimal Lie-Rinehart algebra $\cU(\cD)$ of the distribution $\cD$. We restrict to the case where the module $\cF$ is locally finitely generated, so that $(M,\cF)$ is a singular foliation in the sense of \cite{AS1}.

We have $\cD \subseteq \cF$ as modules and $I_x\cD \subseteq I_x\cF$ as ideals, for every $x \in M$. Whence, by taking the quotients, we find that there is a map $\iota_x : \cD_x \to \cF_x$. This map is not injective, 
 but we can dualize it to obtain a linear map $$\iota^{\ast}_x : \cF_x^{\ast} \to \cD_x^{\ast}.$$ 

\begin{lemma}
The map $\iota^{\ast} : \cF^{\ast} \to \cD^{\ast}$ is continuous.
\end{lemma}
\begin{proof}
As we recalled in Section \ref{sec:dual}, 
given a smooth distribution $(M,\cB)$, the space $\cB^{\ast}=\bigcup_{x \in M}\cB_x^{\ast}$ is a locally compact space when it is endowed with the smallest topology making the projection $p : \cB^{\ast} \to M$ as well as the maps $q_X : \cB^{\ast} \to \R$ continuous, for every $X \in \cB$. It is easy to see that the map $\iota^{\ast}$ commutes with the projections $p^{\cF}$ and $p^{\cD}$ of $\cF^{\ast}$ and $\cD^{\ast}$ respectively, namely $p^{\cD} \circ \iota^{\ast}=p^{\cF}$. Moreover, if $X \in \cD$ then $q^{\cD}_X \circ \iota^{\ast} = q^{\cF}_{\iota(X)}$. (Here $q^{\cD}_X : \cD^{\ast} \to \R$ and $q^{\cF}_X : \cF^{\ast} \to \R$ are the maps induced by the vector fields $X \in \cD$ and $\iota(X) \in \cF$ respectively.) Whence $\iota^{\ast}$ is continuous.
\end{proof}

The operator $\Delta_{\cD}$ also defines a second order pseudodifferential multiplier in the longitudinal pseudodifferential calculus associated with the singular foliation $\cF$, which was constructed in \cite{AS2}. 

Explicitly, recall that any vector field $X\in \cF$ is a first order differential multiplier, whose longitudinal principal symbol is a continuous function on $\cF^*$ is given by 
\[
\sigma_X(x,\xi)=\langle [X]_x, \xi\rangle, \quad x\in M, \xi\in \cF^*_x. 
\] 
Using properties of the principal symbol and the ``sum of squares'' description \eqref{e:sum_of_sq}, we can compute the longitudinal principal symbol of $\Delta_{\cD}$ as follows:
\begin{multline*}
\sigma_{\Delta_{\cD}}(x,\xi)
=\sum_{i=1}^d \left|\langle \hat\rho_{U,x}[\omega_i(x)], \xi\rangle \right|^2=\sum_{i=1}^d \left|\langle \hat\rho_{U,x}[\omega_i(x)], \iota^*_x\xi\rangle \right|^2\\ =\sum_{i=1}^d \left|\langle \omega_i(x), \hat\rho^\ast_{U,x}[\iota^*_x\xi]\rangle \right|^2=\left|\hat\rho^\ast_{U,x}[\iota^*_x\xi]\right|_{E^*_{U,x}}^2= \left|\iota^*_{x}(\xi) \right|^2_{\cD^*_{x}}, \quad x\in M, \xi\in \cF^*_x.
\end{multline*}
Here at the last step, we used the fact that $\hat\rho^\ast_{U,x} : \cD^*_{x}\to E^*_{U,x}$ is an isometry.
\begin{remark}
Notice that this symbol vanishes outside the zero section of $\cF^{\ast}$. Specifically, it vanishes on the subset $\coprod_{x \in M}\{\xi \in \cF^{\ast}_x : \xi\left|_{\iota_x(\cD_x)}\right.=0\}$. Whence, our operator $\Delta_{\cD}$ may not be elliptic in the longitudinal pseudodifferential calculus of $(M,\cF)$.
\end{remark}

\subsection{The horizontal Laplacian as a multiplier of the foliation algebra}\label{sec:multiplier}

Let $(M,\cD)$ be a smooth distribution such that $\cF = \cU(\cD)$ is a foliation. We show here the existence of a pseudodifferential multiplier $P_{\cD}$ of $C^{\ast}_r(\cF)$, in the sense of \cite{AS2}, such that the horizontal Laplacian $\Delta_{\cD}$ is the representation of $P_{\cD}$ to $L_2(M,\mu)$.


Indeed, as shown in Lemma A.2, our Laplacian can be written as
\[
\Delta_\cD=\sum_{\alpha=1}^m\sum_{j=1}^{d_\alpha} \phi_\alpha (X^{(\alpha)}_j)^* X^{(\alpha)}_j\psi_\alpha  
\]
or just 
\[
\Delta_\cD=\sum_{j=1}^m Y_j^* X_j  
\]
with some $X_j, Y_j\in \cD$. Now, from [3] (or [10]) we know that each $X\in \cF$ is the presentation of some multiplier $X^\cF\in \Psi^1(\cF)$ and, since the presentation is a $\ast$-presentation, each $X^*\in \cF$ is the presentation of the multiplier $(X^\cF)^*\in \Psi^1(\cF)$. Therefore, $\Delta_\cD$ is the presentation of the multiplier 
\[
P_\cD=\sum_{j=1}^m (Y^\cF_j)^* X^\cF_j\in \Psi^2(\cF).   
\]
Note that the above also works for noncompact manifolds, because in this case all the sums are infinite, but locally finite.

\begin{remark}
The proof of hypoellitipicity for $\Delta_{\cD}$ that we give in \S \ref{sec:hypoell} goes through verbatim for $P_{\cD}$ as well. In order to prove the essential self-adjointness of $P_{\cD}$ as we do in \S \ref{sec:selfadj} though, one needs to generalize the results in \cite{Chernoff} to the setting developed in \cite{AS2}. This is beyond the scopes of the current article.
\end{remark}

\section{Examples}

Here we present explicit examples of the constructions given in the previous sections. Specifically, we provide explicit calculations for the Riemannian metric of a distribution $(M,\cD)$, the horizontal differential $d_{\cD}$ and its dual $d_{\cD}^{\ast}$, as well as the Laplacian $\Delta_{\cD}$ (verifying that it is a sum of squares), in the following cases: First, in \S \ref{sec:vanishorigin} we look at the distribution $(\R^2,\cD)$ where $\cD$ is the module of vector fields in $\R^2$ which vanish at the origin. In other words, the module $\cD$ in this case is the one generated by the infinitesimal generators of the action of $GL(2,\R)$ on $\R^2$. Second, in \S \ref{sec:pathological}, we examine the quite pathological distribution of $\R^2$ mentioned in item e) of examples \ref{exs:distr}. Third, in \S \ref{sec:Heisenberg} we consider the sub-Riemannian structure of the Heisenberg group.

Notice that our first example arises from a Lie group action. More generally, let $\g$ be a Lie algebra of dimension $k$ and $\g \to \cX(M), V \mapsto V^{\dagger}$ be an (infinitesimal) action of $\g$ on a smooth manifold $M$.
Put $\cD$ the submodule of vector fields generated by all vector fields $V^{\dagger}$ with $V\in \g$. In fact, $\cD$ in this case is a foliation. In the case $G$ is the Lie algebra of a compact Lie group $G$ and $M$ is compact, any invariant Riemannian metric on $G$ gives rise to a Riemannian metric on $\cD$ and the associated horizontal Laplacian $\Delta_{\cD}$ is exactly the operator $-\Delta_G$ introduced by Atiyah in \cite[page 12]{AtiyahLNM401}. Also note that the construction of $\Delta_{\cD}$ does not require any compactness assumptions. (Of course, neither does the construction of $-\Delta_G$.) 

On the other hand, the distribution $(\R^3,\cD)$ arising considering the sub-Riemannian structure of the Heisenberg group, is not involutive. However, the fibers $\cD_{(x,y,z)}$ have dimension $2$ at every $(x,y,z) \in \R^3$. Whence $\cD$ is a projective module, and the familiar Serre-Swan theorem implies that it is the module of sections of a vector sub-bundle $H \to \R^3$ of $T\R^3$. This bundle is a minimal local presentation of $\cD$, where $\rho : H \to TM$ is the inclusion map. This is the case for any smooth distribution $(M,\cD)$ such that the module $\cD$ is projective. It follows that, in cases as such (e.g. the Heisenberg group), our horizontal Laplacian $\Delta_{\cD}$ coincides with the one given in \cite{YK1} and, in the case when the distribution is bracket generating, it coincides with the usual sub-Laplacian in sub-Riemannian geometry (see, for instance, \cite{Agrachev2009}, \cite{Gordina}, \cite{Hassannezhad}, \cite{Montgomery} and the references therein)

Last, the module $\cD$ of the pathological distribution we examine in \S \ref{sec:pathological} is neither projective, nor a foliation. Nevertheless, we are able to attach a horizontal Laplacian to it.

\subsection{Vector fields on the plane, vanishing at the origin}\label{sec:vanishorigin}





Let us consider the distribution $(\mathbb R^2, \cD)$, where $\cD$ is the $C^\infty_c(\mathbb R^2)$-module of compactly supported vector fields on $\RR^2$, vanishing at the origin. In fact, this is the foliation generated by vector fields 
\[
X_{11}=x\partial_x,\quad X_{12}=x\partial_y,\quad X_{21}=y\partial_x,\quad X_{22}=y\partial_y.
\]
Working as in examples \ref{exs:fibcalc} we find $\cD_{(x,y)}\cong \RR^2$ if $(x,y)\neq (0,0)$ and $\cD_{(0,0)}\cong \RR^4$.

Consider $\alpha=\alpha_1(x,y)dx+\alpha_2(x,y)dy\in \Omega^1_c(\RR^2)$ and recall that $ev^{\ast}(\alpha)(x,y)([X]_{(x,y)}) = \alpha_{(x,y)}(X)$ for every $X \in \cD$. So we have
\[
ev^*(\alpha)(x,y)=(\alpha_1(x,y), \alpha_2(x,y))\in \cD^*_{(x,y)}\cong \RR^2 \text{ if } (x,y)\neq (0,0)
\] 
and 
\[
ev^*(\alpha)(0,0)=0\in \cD_{(0,0)}\cong \RR^4. 
\]
Whence, for $f\in C^\infty_c(\RR^2)$, $d_\cD f(x,y)=ev^{\ast}(df)(x,y)=(\partial_xf(x,y), \partial_yf(x,y))\in \cD^*_{(x,y)}\cong \RR^2$ if $(x,y)\neq (0,0)$ and $d_\cD f(0,0)=0\in \cD_{(0,0)}\cong \RR^4$.

The minimal local presentation $E_U$ at $(0,0)$ is given by the trivial vector bundle $E_U=\RR^2\times \RR^4$ over $U=\RR^2$. If we denote by $\{\sigma_{ij}, i,j=1,2\}$ the standard base in $\mathbb R^4$ and by $\{\sigma^*_{ij}, i,j=1,2\}$ the dual base in $(\mathbb R^4)^*$, then $\rho_U$ sends each $\sigma_{ij}$ to $X_{ij}$. Now, for $\alpha=\alpha_1(x,y)dx+\alpha_2(x,y)dy\in \Omega^1_c(\RR^2)$ we find $\langle x\partial_x , \alpha \rangle=x\alpha_1(x,y)$ and $\langle x\partial_y , \alpha \rangle=x\alpha_2(x,y)$, $\langle y\partial_x , \alpha \rangle=y\alpha_1(x,y)$, $\langle y\partial_y , \alpha \rangle=y\alpha_2(x,y)$. Therefore
\[
\rho^*_U\alpha=x\alpha_1\sigma^*_{11}+x\alpha_2\sigma^*_{12}+y\alpha_1\sigma^*_{21}+y\alpha_2\sigma^*_{22}
\]
Whence, for $f\in C^\infty_c(\RR^2)$, $d_{E_U} f\in C^\infty_c(\RR^2, E^*_U)$ is given by
\[
d_{E_U} f(x,y)=xf_x(x,y)\sigma^*_{11}+xf_y(x,y)\sigma^*_{12}+yf_x(x,y)\sigma^*_{21}+yf_y(x,y)\sigma^*_{22}.
\]

The restriction of a Riemannian metric on $\cD$ to $\mathbb R^2\setminus \{0\}$ is a Riemannian metric on the manifold $\mathbb R^2\setminus \{0\}$, in other words, a smooth family of inner products on the fibers of the trivial bundle $T(\R^2 \setminus \{0\})=(\R^2 \setminus \{0\}) \times \R^2$. 
So, it can be written as 
\[
g_{(x,y)}=A(x,y)dx^2+2B(x,y)dx\,dy+C(x,y)dy^2, \quad (x,y)\neq (0,0). 
\]
with some $A,B,C\in C^\infty_c(\R^2 \setminus \{0\}).$ Its behavior near the origin is described as follows. Let $\{G_{(x,y)}, (x,y)\in \RR^2\}$ be a smooth family of inner products in the fibers of $E_U$:
\[
G_{(x,y)}=\sum_{i_1,j_1,i_2,j_2=1,2} G_{i_1j_1,i_2j_2}(x,y)\sigma^*_{i_1j_1}\sigma^*_{i_2j_2},
\]
then, for any $(x,y)\in \RR^2$, the map $\rho_U : \mathbb R^4\to \RR^2$ is a Riemannian submersion, or, equivalently, $\rho^*_U : (\mathbb R^2)^*\cong T^*_{(x,y)}\RR^2 \to (\RR^4)^*$ is an isometry. For $\alpha=\alpha_1(x,y)dx+\alpha_2(x,y)dy\in \Omega^1_c(\RR^2)$, we have
\[
\|\alpha(x,y)\|^2_{g^{-1}}=\|\rho^*_U\alpha(x,y)\|^2_{G^{-1}}. 
\]
In particular, if $G$ is the standard metric on $\RR^4$, then $(\sigma_{ij}, i,j=1,2)$ is an orthonormal base in $\mathbb R^4$ and
\[
\|\alpha(x,y)\|^2_{g^{-1}}=(x^2+y^2)(\alpha_1^2(x,y)+\alpha_2^2(x,y)). 
\]
We get
\[
g_{(x,y)}=\frac{1}{x^2+y^2}(dx^2+dy^2), \quad (x,y)\neq (0,0). 
\]
Assume that the positive density $\mu$ on $\RR^2$ is given by
\[
\mu=dx\,dy.
\]

Let $\omega$ be a map $\RR^2 \ni (x,y) \mapsto \omega(x,y) \in \cD^{\ast}_{(x,y)}$. By definition, $\omega$ is a smooth section of $\cD^{\ast}$ iff its local realization $\omega_U$ defined by  $\omega_U=\widehat\rho^{\ast}_{U} \circ \alpha$ is smooth on $\RR^2$. If we write $\omega$ on $\RR^2\setminus\{0\}$ as $\omega=\omega_1(x,y)dx+\omega_2(x,y)dy$, then 
\[
\widehat \rho^*_U\omega=x\omega_1\sigma^*_{11}+x\omega_2\sigma^*_{12}+y\omega_1\sigma^*_{21}+y\omega_2\sigma^*_{22},
\]
and $\omega$ is smooth iff the functions $x\omega_1, x\omega_2, y\omega_1, y\omega_2$ extend to smooth functions on $\RR^2$.  

For $\omega\in C^\infty_c(\RR^2,\cD^*)$ of the form $\omega=\omega_1(x,y)dx+\omega_2(x,y)dy$ on $\RR^2\setminus\{0\}$, by definition, we have 
\begin{multline*}
\int_{\RR^2} d^{\ast}_{\cD}\omega(x,y)f(x,y)\,dx\,dy= \int_{\RR^2} \langle \omega(x,y), d_{\cD}f(x,y)\rangle_{\cD^*_{(x,y)}}\,dx\,dy\\ =\int_{\RR^2} (x^2+y^2)\left(\omega_1(x,y)\frac{\partial f}{\partial x}(x,y)+\omega_2(x,y)\frac{\partial f}{\partial y}(x,y)\right)\,dx\,dy  
\end{multline*} 
for every $f \in C^{\infty}_c(\R^2)$. Integrating by parts, we see that $d^*_\cD\omega\in C^\infty_c(\RR^2,\cD^*)$ must be given by  
\[
d^*_\cD\omega(x,y)=-\frac{\partial}{\partial x}((x^2+y^2)\omega_1)-\frac{\partial}{\partial y}((x^2+y^2)\omega_2).
\]
Finally, for $f\in C^\infty_c(\RR^2)$, $\Delta_\cD f\in C^\infty_c(\RR^2)$ is given by 
\[
\Delta_\cD f(x,y)=-\frac{\partial}{\partial x}\left((x^2+y^2)\frac{\partial f}{\partial x}\right)-\frac{\partial}{\partial y}\left((x^2+y^2)\frac{\partial f}{\partial y}\right).
\]

In accordance with item c), this operator admits a ``sum of squares'' description:
\[
\Delta_\cD=X^*_{11}X_{11}+X^*_{12}X_{12}+X^*_{21}X_{21}+X^*_{22}X_{22}.
\]
Notice that the above expression shows that the horizontal Laplacian $\Delta_\cD$ is a longitudinal Laplacian of $(M,\cD)$ introduced in \cite{AS2}.

\begin{remark}\label{rem:regul1}
This example also illustrates the kind of regularity represented by the algebra $C^{\infty}_c(M,\cD^{\ast})$ in general. As we already pointed out, in this particular example, an element $\omega$ of $C^{\infty}(\R^2,\cD^{\ast})$ is a map $(\omega_1,\omega_2) : \R^2\setminus\{0\} \to \R^2$ such that the functions  $x\omega_1, x\omega_2, y\omega_1, y\omega_2 : \R^2 \setminus \{0\} \to \R$ extend to smooth functions on $\R^2$. This is equivalent to the functions $\omega_1,\omega_2 : \R^2 \setminus \{0\} \to \R$ being smooth in the usual sense. We also have $\omega(0,0)= \widehat \rho^*_U\omega(0,0) =0\in \cD^{\ast}_{(0,0)}\cong \R^4$. 
\end{remark}

\begin{remark}
As we already said, the module $\cD$ in this example is generated from the infinitesimal generators of the action of $GL(2,\R)$ on $\R^2$. Recall that the foliation associated with this action has also been considered in \cite{AS1}. In fact, the horizontal Laplacian $\Delta_{\cD}$ is the longitudinal Laplacian introduced in \cite{AS2} for this example.

However, the analysis of the Riemannian metric we give here adds some extra information concerning the nature of the singularity at zero. Recall that in \cite{AS1}, the singularity was reflected only by the dimension jump of the fibers $\cD_{(x,y)}$ at $(0,0)$: When $(x,y) \neq (0,0)$ we have $\cD_{(x,y)}=\R^2$, while $\cD_{(0,0)} = \R^4$ is the Lie algebra of $GL(2,\R)$. But now we see that the pathology of the singularity at $(0,0)$ reflects also on the norm of the vectors of $\cD_{(0,0)}$, starting from the Euclidean metric on the local presentation $\R^2 \times \R^4$: Our description of the metric near $(0,0)$ implies that the norm of any vector in $\cD_{(0,0)}$ is none other than infinity.

In other words, even if we start from something as simple as the Euclidean metric of $\R^4$ (which is used to define the metric of the local presentation $\R^2 \times \R^4$) we obtain a Riemannian metric on the fibers of $\cD$ which explodes to infinity at $(0,0)$. Remarkably though, a horizontal Laplacian can still be defined in a geometric way.
\end{remark}

\subsection{The pathological distribution on the plane}\label{sec:pathological}


Here we consider the distribution $(\R^2,\cD)$ discussed in item (e) of examples \ref{exs:distr}. Recall that the module $\cD$ is generated by the vector fields $\partial_x$ and $\phi\partial_y$, where $\phi : \R^2 \to \R$ is defined by $\phi(x,y)=e^{-\frac{1}{x}}$ if $x > 0$ and $\phi(x,y)=0$ if $x \leq 0$. Also recall from item (c) in examples \ref{exs:fibcalc} that its fibers are $\cD_{(x,y)}=\R$ if $x<0$, $\cD_{(0,0)}=\R^2$ and $\cD_{(x,y)}=\R^2$ if $x > 0$.

As in the previous example, let $\alpha = \alpha_1(x,y)dx + \alpha_2(x,y)dy \in \Omega^1(\R^2)$. We find: 
\[
ev^{\ast}(\alpha)(0,y) = (\alpha_1(0,0),\alpha_2(0,y)) \in \R^2, \text{ for any } y \in \R 
\]
\[
ev^{\ast}(\alpha)(x,y) = \alpha_1(x,y) \in \R, \text{ if } x <0,
\]
\[
ev^{\ast}(\alpha)(x,y) = (\alpha_1(x,y),\alpha_2(x,y)) \in \R^2, \text{ if } x >0.
\]
So, if $f \in C^{\infty}_c(\R^2)$ we find 
\[
d_{\cD}f(0,y)=\left(\frac{\partial f}{\partial x}(0,y), \frac{\partial f}{\partial y}(0,y)\right) \text{ for any } y \in \R,
\]
\[
d_{\cD}f(x,y)=\frac{\partial f}{\partial x}(x,y) \text{ if } x < 0, y \in \R, 
\]
and
\[
d_{\cD}f(x,y)=\left(\frac{\partial f}{\partial x}(x,y),\frac{\partial f}{\partial y}(x,y)\right) \text{ if } x > 0, y \in \R.
\]
Now we consider $U=\R^2$ and the local presentation $E_U = \R^2 \times \R^2$ which is minimal at any $(x,y)$ with $x\geq 0$. Again, we will consider the standard Euclidean metric $G$ on $\R^2$, the standard orthonormal frame $\{\sigma_1,\sigma_2\}$ of $E_U$ defined by the canonical (orthonormal) basis of $\R^2$, as well as its dual frame $\{\sigma_1^{\ast},\sigma_2^{\ast}\}$ of $E_U^{\ast}$. The map $\rho_U$ sends $\sigma_1 \mapsto \partial_x$ and $\sigma_2 \mapsto \phi\partial_y$. For an arbitrary 1-form $\alpha = \alpha_1(x,y)dx + \alpha_2(x,y)dy$ we find 
\[
\langle \partial_x,\alpha \rangle = \alpha_1(x,y) \text{ and } \langle \alpha,\phi\partial_y \rangle = \phi(x,y)\alpha_2(x,y).
\]
It follows that $\rho^{\ast}_{U}(\alpha) = \alpha_1(x,y)\sigma_1^{\ast} + \phi(x,y)\alpha_2(x,y)\sigma_2^{\ast}$. Therefore, the local presentation of $d_{\cD}$ is 
\[
d_{E_U}f(x,y) = \frac{\partial f}{\partial x}(x,y) \sigma_1^{\ast} + \phi(x,y)\frac{\partial f}{\partial y}(x,y) \sigma_2^{\ast}.
\]

A map $\omega : (x,y)\in \R^2\mapsto \omega(x,y)\in \cD^*_{(x,y)}$ can be written as $\omega=\omega_1(x,y)\in \cD^*_{(x,y)}\cong \R$ if $x<0$ and $\omega=(\omega_1(x,y),\omega_2(x,y))\in \cD^*_{(x,y)}\cong \R^2$ if $x\geq 0$. It is smooth if and only if $\omega_1\in C^\infty(\R^2)$ and the function $\phi\omega_2$ on $C^\infty(\R^2_+)$ extended by zero to $\cR^2_-$ is smooth on $\R^2$. Whence, if $g$ is the Riemannian metric of $\cD$ induced by $G$, the equality $\|\omega(x,y)\|_{g^{-1}}^2 = \|\widehat{\rho}_U^{\ast}\omega(x,y)\|_{G^{-1}}^2$ implies that for every $(x,y) \in \R^2$ we get
\[
\|\omega(x,y)\|_{g^{-1}}^2 = \omega_1^2(x,y) 
\]
if $x<0$ and
\[
\|\omega(x,y)\|_{g^{-1}}^2 = \omega_1^2(x,y) + (\phi(x,y))^2\omega_2^2(x,y)
\]
if $x\geq 0$. 
Last, the integration by parts argument discussed in \S \ref{sec:vanishorigin} gives
\[
d_{\cD}^{\ast}\omega(x,y) = -\frac{\partial\omega_1}{\partial x}(x,y)
\]
if $x<0$  and 
\[
d_{\cD}^{\ast}\omega(x,y) = -\frac{\partial\omega_1}{\partial x}(x,y) - \frac{\partial}{\partial y}\left(\phi(x,y)^2\omega_2(x,y)\right)
\]
if $x\geq 0$. 
Finally
\[
\Delta_{\cD}f(x,y) = -\frac{\partial^2 f}{\partial x^2} -\frac{\partial}{\partial y}\phi(x,y)^2\frac{\partial f}{\partial y}.
\]

\begin{remark}\label{rem:regul2}
As in remark \ref{rem:regul1}, here we point out the kind of regularity represented by the algebra $C^{\infty}_c(M,\cD^{\ast})$ in this case: Let $\omega\in C^{\infty}_c(\R^2,\cD^{\ast})$. Then the restriction of $\omega$ to the left half-plane $U_{-} = \{(x,y) \in \R^2 : x < 0\}$ is a smooth map $(x,y)\in U_{-} \mapsto \omega_1^-(x,y) \in \R$ in the usual sense. The restriction of $\omega$ to the closed right half-plane $\bar U_{+} = \{(x,y) \in \R^2 : x\geq 0\}$ is a smooth map $(x,y)\in U_{+} \mapsto (\omega^+_1(x,y), \omega^+_2(x,y)) \in \R^2$. Finally, we have compatibility conditions: the function $\omega_1$, which is equal to $\omega_1^-$ on $U_-$ and $\omega_1^+$ on $\bar U_+$, is a smooth function on $\R^2$, and the function $\phi(x,y)\omega_2(x,y), x>0$ extended by zero to $\R^2$ is a smooth function on $\R^2$. For instance, we can take $\omega_2(x,y)=e^{\alpha/x}, (x,y)\in U_+$ with $\alpha<1$.
\end{remark}

\subsection{The Heisenberg group}\label{sec:Heisenberg}


Consider the distribution $(\R^3,\cD)$, where the module $\cD$ is generated by the vector fields
\[
X = \partial_x - \frac{1}{2}y\partial_z, \qquad Y = \partial_y + \frac{1}{2}x\partial_z
\]
We have $[X,Y]=\partial_z$ (also $[X,\partial_z]=[Y,\partial_z]=0$), so $\cD$ is not involutive. Moreover, the vector fields $X, Y$ are linearly independent (with respect to $C^{\infty}(\R)$-coefficients), so the module $\cD$ is projective. Whence, for every $(x,y,z) \in \R^3$ the fiber $\cD_{(x,y,z)}$ is isomorphic to $\R^2$, therefore $H = \cup_{(x,y,z)\in \R^3}\cD_{(x.y.z)}$ is a rank $2$ vector subbundle of $T{\R^3}$. Similarly, for every $(x,y,z) \in \R^3$ the fiber $\cD^*_{(x,y,z)}$ is isomorphic to $\R^2$, and $H^* = \cup_{(x,y,z)\in \R^3}\cD^*_{(x.y.z)}$ is a rank $2$ vector subbundle of $T{\R^3}$.

Given a $1$-form $\alpha = \alpha_1(x,y,z)dx + \alpha_2(x,y,z)dy + \alpha_3(x,y,z)dz$ in $\Omega^1(\R^3)$, for every $(x,y,z) \in \R^3$ we find:
\[
ev^{\ast}(\alpha)(x,y,z) = \left(\alpha_1(x,y,z) - \frac{1}{2}\alpha_3(x,y,z)y, \alpha_2(x,y,z) + \frac{1}{2}\alpha_3(x,y,z)x\right)\in \cD^*_{(x,y,z)}\cong \R^2. 
\]
Whence for every $f \in C^{\infty}_c(\R^3)$ we have:
\[
d_{\cD}f = \left(\frac{\partial f}{\partial x}(x,y,z) - \frac{1}{2}\frac{\partial f}{\partial z}(x,y,z)y, \frac{\partial f}{\partial y}(x,y,z) + \frac{1}{2}\frac{\partial f}{\partial z}(x,y,z)x\right)\in \cD^*_{(x,y,z)}\cong \R^2. 
\]

Put $U = \R^3$, $E_{U} = \R^3 \times \R^2$ and consider the standard Euclidean metric $G$ on $\R^2$ and the standard orthonormal frame $\{\sigma_1, \sigma_2\}$ of $E_U$ induced by the canonical orthonormal basis of $\R^2$, as well as its dual frame $\{\sigma_1^{\ast}, \sigma_2^{\ast}\}$. The map $\rho_U : E_U \to T\R^3$ sends $\sigma_1 \mapsto X$ and $\sigma_2 \mapsto Y$. We find:
\[
ev^{\ast}(\alpha)(x,y,z)([X]_{(x,y,z)}) = \alpha_1(x,y,z) - \frac{1}{2}y \alpha_3(x,y,z)
\]
\[
ev^{\ast}(\alpha)(x,y,z)([Y]_{(x,y,z)}) = \alpha_2(x,y,z) + \frac{1}{2}x\alpha_3(x,y,z)
\]
It follows that 
\[
\rho_U^{\ast}(\alpha)(x,y,z) = \left(\alpha_1(x,y,z) - \frac{1}{2}y\alpha_3(x,y,z)\right)\sigma_1^{\ast} + \left( \alpha_2(x,y,z) + \frac{1}{2}x\alpha_3(x,y,z)\right)\sigma_2^{\ast}
\]
so the local presentation of $d_{\cD}$ is 
\[
d_{E_U^{\ast}}f(x,y,z) = \left(\frac{\partial f}{\partial x}(x,y,z) - \frac{1}{2}y\frac{\partial f}{\partial z}(x,y,z)\right)\sigma_1^{\ast} + \left( \frac{\partial f}{\partial y}(x,y,z) + \frac{1}{2}x\frac{\partial f}{\partial z}(x,y,z)\right)\sigma_2^{\ast}.
\]
Putting $g$ the Riemannian metric of $\cD$ induced by $G$, for a map $\omega : (x,y,z)\in \R^3\mapsto \omega(x,y,z)=(\omega_1(x,y,z),\omega_2(x,y,z))\in \cD^*_{(x,y,z)}\cong \R^2$ the equality $\|\omega(x,y,z)\|_{g^{-1}}^2 = \|\widehat{\rho}_U^{\ast}\omega(x,y,z)\|_{G^{-1}}^2$ implies that for every $(x,y,z) \in \R^3$ we have
\[
\|\omega(x,y,z)\|_{g^{-1}}^2=\omega_1(x,y,z)^2+\omega_2(x,y,z)^2.
\]
Whence for every $\omega : (x,y,z)\in \R^3\mapsto \omega(x,y,z)=(\omega_1(x,y,z),\omega_2(x,y,z))\in \cD^*_{(x,y,z)}\cong \R^2$ we have:
\[
d^*_{\cD}\omega(x,y,z)= -\left(\frac{\partial }{\partial x}- \frac{1}{2}\frac{\partial }{\partial z}y\right)\omega_1(x,y,z)-\left(\frac{\partial }{\partial y} + \frac{1}{2}\frac{\partial }{\partial z}x\right)\omega_2(x,y,z). 
\]
Finally, we get
\[
\Delta_{\cD} = -\left(\frac{\partial }{\partial x} - \frac{1}{2}\frac{\partial }{\partial z}y\right)^2-\left(\frac{\partial }{\partial y} + \frac{1}{2}\frac{\partial }{\partial z}x\right)^2.
\]



\section{Some analytic properties of the Laplacian}\label{sec:anal}

\subsection{Essential self-adjointness of the Laplacian}\label{sec:selfadj}

In this section we restrict to distributions $(M,\cD)$ such that $M$ is a compact manifold. In this setting we are able to prove the next, fundamental property of our Laplacian.

\begin{thm} 
Let $(M,\cD)$ be a smooth distribution such that $M$ is compact. The Laplacian $\Delta_\cD$, considered as an unbounded operator in the Hilbert space $L^2(M,\mu)$, with domain
$C^\infty(M)$, is essentially self-adjoint.
\end{thm}
 
\begin{proof}
Let $M=\bigcup_{\alpha=1}^k U_\alpha$ be a finite open covering of $M$ such that, for any $\alpha=1,\ldots,k$, there exist a local presentation $(E_{U_\alpha},\rho_{U_\alpha})$ and a local orthonormal frame $(\omega^{(\alpha)}_1,\ldots,\omega^{(\alpha)}_{d_\alpha})$ of $E_{U_\alpha}$. As mentioned above, the restriction of $\Delta_\cD$ to $U_\alpha$ is written as
\[
\Delta_{\cD,U_\alpha}=\sum_{j=1}^{d_\alpha}(X^{(\alpha)}_j)^* X^{(\alpha)}_j,
\]
where $X^{(\alpha)}_j=\rho_{U_\alpha}(\omega^{(\alpha)}_j)\in \cD\left|_{U_\alpha}\right.$, $j=1,\ldots, d_\alpha$.

Take a partition of unity subordinate to this covering, that is, a family $\{\varphi_\alpha\in C^\infty(M), \alpha =1, \ldots, k\}$ of smooth functions on $M$ such that $0\leq \varphi_\alpha(x) \leq 1$ for any $x\in M$, $\operatorname{supp}\varphi_\alpha\subset U_\alpha$ and $\sum_{\alpha=1}^k\varphi^2_\alpha(x)=1$ for any $x\in M$. 

For $N=\sum_{\alpha=1}^kd_\alpha$, denote by $C^\infty(M,\mathbb C^N)$ the space of smooth functions on $M$ with values in the standard Hermitian space $\mathbb C^N$, whose elements are written as $\{v^{(\alpha)}_j, \alpha=1,\ldots,k, j=1,\ldots, d_\alpha\}$. 
Denote by $L^2(M,\mathbb C^N,\mu)$ the associated Hilbert space of square integrable functions. Consider the operator $D: C^\infty(M)\to C^\infty(M,\mathbb C^N)$ given, for $u\in C^\infty(M)$, by
\[
(Du)^{\alpha}_j=  X^{(\alpha)}_j(\varphi_\alpha u), \quad \alpha=1,\ldots,k , \quad j=1,\ldots, d_\alpha.
\]

On the Hilbert space $H=L^2(M,\mu)\oplus L^2(M,\mathbb C^N,\mu)$, consider the operator $A$, with domain $D(A)=C^\infty(M)\oplus
C^\infty(M,\mathbb C^N)$, given by the matrix
\[
A=\begin{pmatrix} 0 & D^*\\
D & 0
\end{pmatrix}.
\]
It is clear that the operator $A$ is symmetric. Applying \cite[Theorem
2.2]{Chernoff} to the skew-symmetric operator $L=iA$, we obtain that
every power of $A$ is essentially self-adjoint. Since
\[
A^2=\begin{pmatrix} D^*D & 0\\
0 & DD^*
\end{pmatrix},
\]
the operator 
\[
D^*D=\sum_{\alpha=1}^m\sum_{j=1}^{d_\alpha} \varphi_\alpha (X^{(\alpha)}_j)^* X^{(\alpha)}_j\varphi_\alpha=\sum_{\alpha=1}^m \varphi_\alpha \Delta_{\cD,U_\alpha} \varphi_\alpha
\]
is essentially self-adjoint on $C^\infty(M)$.

Now we use the IMS localization formula:
\[
\Delta_{\cD}=\sum_{\alpha=1}^k \varphi_\alpha \Delta_{\cD,U_\alpha}\varphi_\alpha+\frac 12 \sum_{\alpha=1}^k [[\Delta_{\cD},\varphi_\alpha],\varphi_\alpha].
\]

Since the operator $\frac 12 \sum_{\alpha=1}^k [[\Delta_{\cD},\varphi_\alpha],\varphi_\alpha]$ is bounded, by the Kato-Rellich theorem (see \cite[Ch. V, Thm 4.4]{Kato}), it follows that the operator $\Delta_{\cD}$ is essentially self-adjoint on $C^\infty(M)$.

\end{proof}

\subsection{Longitudinal hypoellipticity of the Laplacian}\label{sec:hypoell}


In this section we prove the hypoellipticity of the horizontal Laplacian $\Delta_{\cD}$, for a distribution on a compact manifold $M$. To this end, we will make substantial use of the viewpoint on $\Delta_{\cD}$ as a longitudinal differential operator. So, throughout this section we fix a smooth distribution $(M,\cD)$ such that $M$ is compact and its minimal Lie-Rinehart algebra $\cF = \cU(\cD)$ is a foliation. Using local presentations of the given distribution, we are able in \S \ref{sec:longhypoell} to follow the line of proof for hypoellipticity given in \cite{YK1}, appropriately adapted to our context.

\subsubsection{Longitudinal pseudodifferential calculus}\label{s:psi}

We will need the classes $\Psi^m({\mathcal F})$ of longitudinal pseudodifferential operators. Operators as such were constructed in \cite{AS2} as multipliers of the foliation $C^{\ast}$-algebra. Here, as in \cite[\S 3]{YK1}, we will consider their image by the trivial representation to $L^2(M,\mu)$. In this section we recall the following results from \cite[\S 3]{YK1}, that are used in \S \ref{sec:longhypoell} in order to prove hypoellipticity.

One can define the longitudinal principal symbol map
$\sigma_m : \Psi^m({\mathcal F})\to C(\mathcal F^*\setminus 0)$. Here $\mathcal F^*$ denotes the cotangent bundle of $\mathcal F$ (see Section \ref{sec:dual}). 

\begin{thm} \label{t:comp}
Given $P_i\in \Psi^{m_i}({\mathcal F})$, $i=1,2$, their composition
$P=P_1\circ P_2$ is in $\Psi^{m_1+m_2}({\mathcal F})$ and
$\sigma_{m_1+m_2}(P)=\sigma_{m_1}(P_1)\sigma_{m_2}(P_2)$.
\end{thm}

\begin{thm}
Given $P_i\in \Psi^{m_i}({\mathcal F})$, $i=1,2$, the commutator
$[P_1, P_2]$ is in $\Psi^{m_1+m_2-1}({\mathcal F})$.
\end{thm}

An operator $P\in \Psi^{m}({\mathcal F})$ is said to be
longitudinally elliptic, if its longitudinal principal symbol
$\sigma_{m}(P)$ is invertible.

\begin{thm} \label{t:reg}
Given  a longitudinally elliptic operator $P\in \Psi^{m}({\mathcal
F})$, there is an operator $Q\in \Psi^{-m}({\mathcal F})$ such that
$1-P\circ Q$ and $1-Q\circ P$ are in $\Psi^{-\infty}({\mathcal F})$.
\end{thm}

For any $s$, we fix a longitudinally elliptic operator $\Lambda_s$ of order $s$. Without loss of generality, we can assume that $\Lambda_s$ is formally self-adjoint and
\[
\Lambda_s\circ \Lambda_{-s}=I+R_s, \quad \Lambda_{-s}\circ
\Lambda_{s}=I+R^\prime_s, \quad R_s, R^\prime_s\in \Psi^{-\infty}(\mathcal F).
\]

\begin{definition}
For $s\geq 0$, the Sobolev space $H^s(\mathcal F)$ is defined as the
domain of $\Lambda_s$ in $L^2(M)$:
\[
H^s(\mathcal F)=\{u\in L^2(M) : \Lambda_su\in L^2(M)\}.
\]
 The norm in $H^s(\mathcal F)$ is
defined by the formula
\[
\|u\|^2_{s}=\|\Lambda_su\|^2+\|u\|^2, \quad u\in H^s(\mathcal F).
\]

For $s<0$, $H^s(\mathcal F)$ is defined as the dual space of
$H^{-s}(\mathcal F)$.
\end{definition}

\begin{thm}\label{t:action_in_Sobolev}
For any $s\in \mathbb R$, an operator $A\in \Psi^m(\mathcal F)$
determines a bounded operator $A : H^s(\mathcal F)\to H^{s-m}(\mathcal F)$.
\end{thm}

\begin{prop}\label{p:density_in_Sobolev}
For $s\in \mathbb Z$, the space $C^\infty(M)$ is dense in
$H^s(\mathcal F)$.
\end{prop}

\subsubsection{Longitudinal hypoellipticity}\label{sec:longhypoell}


As above, let $M$ be a compact manifold and $(M,\cD)$ be a smooth distribution such that $\cF=\cU(\cD)$ is a foliation. Let $g$ be a Riemannian structure on $\cD$ and $\mu$ a positive smooth density on $M$. We will use classes $\Psi^m({\mathcal F})$ of longitudinal pseudodifferential operators and the corresponding scale $H^s(\mathcal F)$ of longitudinal Sobolev space associated with $\mathcal F$ (\cf \S \ref{s:psi}).

As in \cite{YK1}, we follow the line of proof of hypoellipticity for sums of squares operators given in \cite[Chapter II, \S 5]{Treves1}. (The specific hypoellipticity result there is \cite[Chapter II, Cor. 5.1]{Treves1}.) In fact, the proof is as the one given in \cite{YK1}, so here we restrict to describing it.

First, as in \cite[Chapter II, Lemma 5.2]{Treves1} we state subelliptic estimates for the operator $\Delta_\cD$. The proof of Theorem \ref{t:Hs-hypo} below, is exactly as the proof \cite[Thm. 2.1]{YK1}, except for two lemmas that need to be adapted to the current setting. We will give these lemmas in Appendix \ref{app}.

\begin{thm}\label{t:Hs-hypo}
There exists $\epsilon>0$ such that, for any $s\in \mathbb R$, we have
\[
\|u\|_{s+\epsilon}^2 \leq
C_s\left(\|\Delta_\cD u\|_s^2+\|u\|_s^2\right), \quad u\in C^\infty(M),
\]
where $C_s>0$ is some constant.
\end{thm}

As a consequence, we get the following longitudinal hypoellipticity result. Again, its proof is exactly as the proof of \cite[Thm. 2.2]{YK1}, so we omit it.

\begin{thm}\label{t:hypo}
If $u\in H^{-\infty}(\mathcal F):=\bigcup_{t\in \mathbb R}H^t(\mathcal F)$ such that $\Delta_\cD u\in H^s(\mathcal F)$ for some $s\in \mathbb R$, then $u \in H^{s+\varepsilon}(\mathcal F)$. 
\end{thm}

\subsubsection{Proof of Theorem \ref{t:Hs-hypo}}\label{app}

As in \cite{YK1}, for the proof of Theorem~\ref{t:Hs-hypo} we follow Kohn's proof of the
subellipticity of the H{\"o}rmander's operators \cite{Kohn} (see also \cite{Treves1},\cite{HelfferNier}). It is easy to see that the proof of \cite[Thm. 2.1]{YK1} given in \cite[\S 4]{YK1}, also holds for the Laplacian $\Delta_{\cD}$ we constructed in this paper. The only points that need to be adapted to our current context are \cite[Lemma 4.1]{YK1} and \cite[Lemma 4.2]{YK1}. So here we just give the proofs of these lemmas for the operator $\Delta_{\cD}$.

Starting with this, \cite[Lemma 4.1]{YK1} is replaced by the next lemma.

\begin{lemma}
For any $X\in \cD$, there exists $C>0$ such that
\begin{equation}\label{Xestimate}
\|Xu\|^2 \leq C\left((\Delta_\cD u,u)+\|u\|^2\right), \quad
u\in C^\infty(M).
\end{equation}
\end{lemma}

\begin{proof}
Let $U$ be an open subset of $M$ such that there exist a local presentation $(E_U,\rho_U)$ and a local
orthonormal frame $(\omega_1,\ldots,\omega_d)$ of $E_U$. Then, for any $u\in C^\infty_c(U)$, we have
\[
(\Delta_{\cD,U}u,u)=\sum_{i=1}^d \int_U |\rho_U[\omega_i]u(x)|^2d\mu(x).
\]
Take an arbitrary $\omega\in \Gamma E_U$ such that $\rho_U(\omega)=X\left|_U\right.$. We can write $\omega=\sum_{j=1}^da_j\omega_j$ with some $a_j\in C^\infty(\bar U), j=1,\ldots,d$. Therefore, for any $u\in C^\infty_c(U)$, we get
\[
\|Xu\|^2=\int_U |\rho_U(\omega) u(x)|^2\,d\mu(x) \leq C\sum_{j=1}^d\int_U |\rho_U[\omega_j]u(x)|^2\,d\mu(x) = C(\Delta_\cD u,u).
\]

To prove the estimate \eqref{Xestimate} in the general case, we take a
finite open covering $M=\cup_{\alpha=1}^k U_\alpha$ of $M$ such that, for any $\alpha=1,\ldots,k$, there exist a local representation $(E_{U_\alpha},\rho_{U_\alpha})$ and a local orthonormal frame $(\omega^{(\alpha)}_1,\ldots,\omega^{(\alpha)}_{d_\alpha})$ of $E_{U_\alpha}$. Take a partition of unity subordinate to this covering, that is, a family $\{\varphi_\alpha\in C^\infty(M), \alpha =1, \ldots, k\}$ of smooth functions on $M$ such that $0\leq \varphi_\alpha(x) \leq 1$ for any $x\in M$, $\operatorname{supp}\varphi_\alpha\subset U_\alpha$ and $\sum_{\alpha=1}^k\varphi^2_\alpha(x)=1$ for any $x\in M$. Now we use the IMS localization formula:
\[
(\Delta_{\cD}u,u)=\sum_{\alpha=1}^k (\Delta_{\cD,U_\alpha}(\varphi_\alpha u), \varphi_\alpha u)+\frac 12 \sum_{\alpha=1}^k ([[\Delta_{\cD,U_\alpha},\varphi_\alpha],\varphi_\alpha] u, u), \quad u\in C^\infty(M),
\]
and the fact that, for any $\varphi\in C^\infty(M)$, the operators $[X,\varphi]$ and $[[\Delta_{\cD},\varphi],\varphi]$ are zero order differential operators and, therefore, bounded in $L^2$.
\end{proof}

Now \cite[Lemma 4.2]{YK1} is replaced by the next lemma.

\begin{lemma}\label{l:com}
The operator $[\Delta_\cD, \Lambda_s]$ can be represented in the form
\[
[\Delta_\cD, \Lambda_s]=\sum_{k=1}^N T^s_k X_k+T^s_0,
\]
where $X_k\in \cD, k=1,\ldots, N,$ and $T^s_k \in
\Psi^s(\mathcal F), k=0,\ldots, N$.
\end{lemma}

\begin{proof}
Let $M=\bigcup_{\alpha=1}^m U_\alpha$ be a finite open covering of $M$ such that, for any $\alpha=1,\ldots,m$, there exist a local representation $(E_{U_\alpha},\rho_{U_\alpha})$ and a local orthonormal frame $(\omega^{(\alpha)}_1,\ldots,\omega^{(\alpha)}_{d_\alpha})$ of $E_{U_\alpha}$. As mentioned above, the restriction of $\Delta_\cD$ to $U_\alpha$ is written as
\[
\Delta_\cD\left|_{U_\alpha}\right.=\sum_{j=1}^{d_\alpha}(X^{(\alpha)}_j)^* X^{(\alpha)}_j,
\]
where $X^{(\alpha)}_j=\rho_{U_\alpha}(\omega^{(\alpha)}_j)\in \cD\left|_{U_\alpha}\right.$, $j=1,\ldots, d_\alpha$.

Let $\phi_\alpha\in C^\infty(M)$ be a partition of unity subordinate
to the covering, ${\rm supp}\,\phi_\alpha\subset U_\alpha$, and
$\psi_\alpha\in C^\infty(M)$ such that ${\rm
supp}\,\psi_\alpha\subset U_\alpha$,
$\phi_\alpha\psi_\alpha=\phi_\alpha$. Then we have
\begin{multline*}
\Delta_\cD=\sum_{\alpha=1}^m \phi_\alpha (\Delta_\cD\left|_{U_\alpha}\right.)\psi_\alpha=\sum_{\alpha=1}^m\sum_{j=1}^{d_\alpha} \phi_\alpha (X^{(\alpha)}_j)^* X^{(\alpha)}_j\psi_\alpha\\ =\sum_{\alpha=1}^m \sum_{j=1}^{d_\alpha} \phi_\alpha (X^{(\alpha)}_j)^* \psi_\alpha X^{(\alpha)}_j + \sum_{\alpha=1}^m \sum_{j=1}^{d_\alpha} \phi_\alpha (X^{(\alpha)}_j)^* [X^{(\alpha)}_j,\psi_\alpha].
\end{multline*}
We can write
\begin{align*}
\phi_\alpha (X^{(\alpha)}_j)^*\psi_\alpha X^{(\alpha)}_j\Lambda_s=& \phi_\alpha (X^{(\alpha)}_j)^*\Lambda_s \psi_\alpha X^{(\alpha)}_j+\phi_\alpha (X^{(\alpha)}_j)^*[\psi_\alpha X^{(\alpha)}_j,\Lambda_s]\\=& \Lambda_s \phi_\alpha (X^{(\alpha)}_j)^*\psi_\alpha X^{(\alpha)}_j+[\phi_\alpha (X^{(\alpha)}_j)^*,\Lambda_s] \psi_\alpha X^{(\alpha)}_j\\ &+[\psi_\alpha X^{(\alpha)}_j,\Lambda_s]\phi_\alpha (X^{(\alpha)}_j)^*+ [\phi_\alpha (X^{(\alpha)}_j)^*,[\psi_\alpha X^{(\alpha)}_j,\Lambda_s]].
\end{align*}
Since $(X^{(\alpha)}_j)^*=-X^{(\alpha)}_j+c^{(\alpha)}_j$ with some $c^{(\alpha)}_j\in C^\infty(M)$, we get
\[
\Delta_\cD \Lambda_s=\Lambda_s \Delta_\cD+\sum_{\alpha=1}^m \sum_{j=1}^{d_\alpha} T^{s,(\alpha)}_{1,j} \psi_\alpha X^{(\alpha)}_j+ \sum_{\alpha=1}^m \sum_{j=1}^{d_\alpha} T^{s,(\alpha)}_{2,j} \phi_\alpha X^{(\alpha)}_j+T^s_0,
\]
where the operators
\begin{align*}
T^{s,(\alpha)}_{1,j}=&[\phi_\alpha (X^{(\alpha)}_j)^*,\Lambda_s], \quad
T^{s,(\alpha)}_{2,j}=-[\psi_\alpha X^{(\alpha)}_j,\Lambda_s], \\
T^s_0=& \sum_{\alpha=1}^m \sum_{j=1}^{d_\alpha} \Big( [\psi_\alpha X^{(\alpha)}_j,\Lambda_s]\phi_\alpha c^{(\alpha)}_j+[\phi_\alpha (X^{(\alpha)}_j)^*,[\psi_\alpha X^{(\alpha)}_j,\Lambda_s]]\\ & +[\phi_\alpha (X^{(\alpha)}_j)^* [X^{(\alpha)}_j,\psi_\alpha],\Lambda_s]\Big)
\end{align*}
belong to $\Psi^s(\mathcal F)$. Setting $\{X_k, k=1,\ldots, N\}=\{\psi_\alpha X^{(\alpha)}_j, \phi_\alpha X^{(\alpha)}_j, \alpha=1,\ldots,m, j=1,\ldots,d_\alpha \}$ with $N=2\sum_{\alpha=1}^md_\alpha$, we complete the proof.
\end{proof}


\appendix



\section{The longitudinal de Rham complex and the Hodge Laplacian}\label{app:deRhamHodge}





The purpose of this appendix is to exhibit that the notion of local presentation, as well as the Riemannian metric we introduce in this paper, can be used to provide further developments for singular situations such as the ones we consider here. Specifically, we present two developments as such:
\begin{itemize}
\item We build the appropriate longitudinal de Rham complex along an arbitrary singular foliation.
\item We construct a Hodge Laplacian for an arbitrary singular foliation.
\end{itemize}
Explicit computations of the longitudinal de Rham cohomology, as well as analytic results arising from the Hodge Laplacian, are the subject of future work.

\subsection{The foliated de Rham complex of a singular foliation}\label{app:longdeRham}

In this section, we consider the case of a generalised smooth distribution $(M,\cF)$ which is involutive, namely it is a singular foliation. In this case, we extend the horizontal differential $d_{\cF} : C^{\infty}(M) \to C^{\infty}(M,\cF^{\ast})$ to a differential complex. It gives rise to an appropriate cohomology of the distribution $(M,\cF)$. This is a version of the foliated cohomology appearing in \cite[\S 2.1]{Sylvain}.

So let us make a fresh start. The following constructions apply to an arbitrary generalised smooth distribution $(M,\cD)$.

\begin{definition}\label{dfn:wedgedistr}
Let $(M,\cD)$ be a generalised smooth distribution and $k \in \N$, $k \geq 1$. We define $\Lambda^k \cD$ to be the $C^{\infty}(M)$-submodule of $\Lambda^k\cX(M)$ generated by $X_1\wedge\ldots\wedge X_k$, where $X_1,\ldots,X_k$ are vector fields in $\cD$. Also put $\Lambda^0 \cD=C^{\infty}(M)$.
\end{definition}
For $k \geq 1$ we make the following easy observations:
\begin{enumerate}
\item An arbitrary element of $\Lambda^k \cD$ is a linear combination of $\sum_{i \in I}\phi_{i}X_{i_1}\wedge\ldots\wedge X_{i_k}$, where $\phi_i \in C^{\infty}(M)$ and $X_{i_1},\ldots,X_{i_k} \in \cD$, for all $i \in I$. 
\item Let $x \in M$ and $U \subset M$ an open neighborhood of $x$. The module $\Lambda^k \cD$ can be restricted to $U$ by putting $(\Lambda^k \cD)|_U = \Lambda^k (\cD|_U)$. 
\item Since $\cD$ is locally finitely generated, $\Lambda^k \cD$ is locally finitely generated as well. 
Put $(\Lambda^k \cD)_x = \frac{\Lambda^k \cD}{I_x\Lambda^k \cD}$. It is easy to see that $(\Lambda^k \cD)_x = \Lambda^k (\cD_x)$. Therefore, we also have $(\Lambda^k \cD)_x^{\ast} = \Lambda^k (\cD_x^{\ast})$. Put $(\Lambda^k \cD)^{\ast} = \bigcup_{x \in M}(\Lambda^k \cD)_x^{\ast}$.
\item Let $\rho_U : E_U \to TM$ be a local presentation of $\cD$ over an open $U\subset M$. Then $\rho_U$ can be extended by linearity to $\Lambda^k\rho_U : \Lambda^k E_U \to \Lambda^k TM$. Put $\widehat{\Lambda^k\rho}_U$ the corresponding map between the respective modules of sections. We have the commutative diagrams:
\begin{eqnarray}\label{eqn:wedge}
\xymatrix{
(\Lambda^k E_U)_x \ar@{->>}[r]^{\widehat{\Lambda^k \rho}_{U,x}} \ar@{->>}[rd]_{\Lambda^k \rho_{U,x}} & \Lambda^k \cD_x \ar@{->>}[d]^{ev_x} \\
& \Lambda^k D_x
}
\quad and \quad
\xymatrix{
 & \Lambda^k \cD^{\ast}_x \ar[d]^{\widehat{\Lambda^k \rho}^{\ast}_{U,x}} \\ 
\Lambda^k T^{\ast}_x M \ar[r]_{\Lambda^k \rho_{U,x}^{\ast}} \ar[ru]|-{ev^{\ast}_x} & \Lambda^k E^{\ast}_{U,x} 
}
\end{eqnarray}
\item We can also define the $C^{\infty}(M)$-module of smooth sections of $\Lambda^k \cD^{\ast}$ as in definition \ref{dfn:smoothdual}. Namely, smooth sections are maps $M \ni x \to \eta^{\ast}(x) \in \Lambda^k \cD^{\ast}_x$, such that: For every $x \in M$ there is a local presentation $(E_U,\rho_U)$ of $\cD$, defined in a neighborhood $U$ of $x$, so that the section of $\Lambda^k E_U^{\ast}$ defined by $\eta_U^{\ast}(y) = \widehat{\Lambda^k\rho}_{U,y} \circ \eta^{\ast}(y)$, is smooth on $U$. We denote this module by $C^{\infty}(M,\Lambda^k\cD^{\ast})$ and write $C^{\infty}_c(M,\Lambda^k\cD^{\ast})$ for the module of sections with compact support. 
\item We can also define these modules of sections in a ``coordinate-free'' way: As in corollary \ref{cor:smoothdual}, we have a bilinear pairing $$C^{\infty}_c(M,\Lambda^k\cD^{\ast}) \otimes_{C^{\infty}_c(M)}\Lambda^k\cD \to C^{\infty}(M).$$
For $\omega \in C^{\infty}_c(M,\Lambda^k\cD^{\ast})$, $X_1\wedge \ldots \wedge X_k\in \Lambda^k\cD$, the function $\langle \omega, X_1\wedge \ldots \wedge X_k\rangle\in C^{\infty}(M)$ is
given by 
\[
\langle \omega, X_1\wedge \ldots \wedge X_k\rangle(x)= \omega(x)([X_1]_x, \ldots, [X_k]_x),\quad x\in M.
\]
\end{enumerate}

Now let us look at the case of a singular foliation. To this end, we change our notation and write $(M,\cF)$ instead of $(M,\cD)$.

\begin{definition}\label{def:hordeRham}
Let $(M,\cF)$ be a singular foliation. 
\begin{enumerate}
\item Elements of $C^{\infty}(M,\Lambda^k \cF^{\ast})$ are called \emph{foliated $k$ forms}.
\item The \emph{foliated de Rham complex} associated with $(M,\cF)$ is
\[
C^{\infty}(M) \stackrel{d^0_{\cF}}{\longrightarrow} C^{\infty}(M,\cF^{\ast}) \stackrel{d^1_{\cF}}{\longrightarrow} C^{\infty}(M,\Lambda^2 \cF^{\ast}) \ldots  C^{\infty}(M,\Lambda^k \cF^{\ast}) \stackrel{d^k_{\cF}}{\longrightarrow} C^{\infty}(M,\Lambda^{k+1} \cF^{\ast}) \stackrel{d^{k+1}_{\cF}}{\longrightarrow} \ldots
\]
where $d^0_{\cF}$ is the longitudinal differential $d_{\cF}$ and for every $k \geq 1$ the differential $d_{\cF}^k$ is given by the usual Chevalley-Eilenberg formula, namely:
\begin{multline}\label{eqn:ChevEil}
d_\cF\eta^{\ast}([X_0]_x,\ldots,[X_k]_x)=\sum_{i=0}^k(-1)^iX_i\langle\eta^{\ast}, X_0\wedge\ldots\wedge\hat{X}_i\wedge\ldots\wedge X_{k}\rangle(x) \\ + \sum_{i<j}(-1)^{i+j}\langle \eta^{\ast}, [X_i,X_j]\wedge X_0\wedge\ldots\wedge\hat{X}_i\wedge\ldots\wedge\hat{X}_j\wedge\ldots\wedge X_{k}\rangle(x).
\end{multline}
for every $\eta^{\ast} \in C^{\infty}(M,\Lambda^k \cF^{\ast})$, $x\in M$  and $[X_0]_x,\ldots,[X_k]_x \in \cF_x$. One can show that this definition is correct, that is, the right hand side is independent of the choice of the representatives $X_0,\ldots,X_k \in \cF$.
\end{enumerate}
\end{definition}


\begin{ex}
Consider the foliation $(\R^2,\cF)$ we discussed in \S \ref{sec:vanishorigin}. It is easy to see that, in this case, the foliated de Rham complex is the de Rham complex of the manifold $\R^2$. Indeed, the restriction of $\cF^{\ast}$ to $\R^2\setminus\{0\}$ is $(\R^2 \setminus\{0\}) \times \R^2$ and the differential operators $d^k_{\cF}$ are the usual de Rham operators. But $\Lambda^3(\R^2)=\Lambda^4(\R^2)=0$, whence $C^{\infty}(\R^2 \setminus\{0\},\Lambda^3(\cF^{\ast})) = C^{\infty}(\R^2 \setminus\{0\},\Lambda^4(\cF^{\ast})) = 0$. On the other hand,  for every $k$, the definition of a smooth section $\omega$ of $\Lambda^k\cF^{\ast}$ near zero uses the minimal local presentation $E_U = \R^2 \times \R^4$. That is to say, the map $\omega_U = \widehat{\Lambda^k\rho}^{\ast} \circ \omega$ must be a smooth section of $E_U$. A continuity argument for $\omega_{U}$ shows that $\omega_U(0)=0$. Passing to the duals and coming back, we find that $\omega(0)=0$ as well, whence $C^{\infty}(\R^2,\Lambda^3(\cF^{\ast}))$ vanishes. The same holds for  $C^{\infty}(\R^2,\Lambda^4(\cF^{\ast}))$. 
\end{ex}

\subsection{The Hodge Laplacian of a singular foliation}\label{app:Hodge}


Having defined the foliated de Rham complex in \S \ref{app:longdeRham}, it is natural to extend the familiar Hodge Laplace operator to singular foliations. Here we sketch its construction.


First note that, given a Riemannian metric on $(M,\cD)$, one can naturally define families of inner products on $\Lambda^{k}\cD_x$ and $\Lambda^{k}\cD^*_x$. To show their smoothness properties, we take a local presentation of the Riemannian metric defined on an open neighborhood $U$, that is, a local presentation $\rho_U : E_U \to TM$ and a smooth family of inner products in the fibers of $E_U$. Then we have smooth families of inner products on the bundles $\Lambda^k E_U$ and $\Lambda^k E^*_U$. One can show that, for every $x \in U$, the map $\widehat{\Lambda^k\rho}_{U,x} : \Lambda^k E_{U,x} \to \Lambda^k\cD_x$ is a Riemannian submersion and the adjoint map $\widehat{\Lambda^k\rho}^*_{U,x} : \Lambda^k\cD^*_x\to \Lambda^k E^*_{U,x}$ is an isometry. As a consequence, we obtain an analogue of Lemma~\ref{lem:D*smooth} for the spaces $C^{\infty}(M,\Lambda^k\cD^{\ast})$.

Now let $(M,\cF)$ be a (singular) foliation equipped with a Riemannian metric on $\cF$ and a smooth positive density on $M$. Using the Riemannian metric on $\Lambda^{k}\cF^{\ast}$ introduced above, we can consider the adjoint $d^{\ast}_{\cF} : C^{\infty}(M,\Lambda^{k+1}\cF^{\ast}) \to C^{\infty}(M,\Lambda^{k}\cF^{\ast})$ of $d_{\cF}$ defined by $$\langle d_{\cF}\alpha,\beta \rangle = \langle \alpha, d^{\ast}_{\cF}\beta \rangle$$ for every $\alpha \in C^{\infty}(M,\Lambda^k\cF^{\ast})$ and $\beta \in C^{\infty}(M,\Lambda^{k+1}\cF^{\ast})$. Its existence can be proved as in Section \ref{sec:hordif}, using local presentations.
\begin{definition}\label{dfn:HodgeLaplace}
The \emph{Hodge-Laplace operator on foliated $k$-forms} is the operator $$\Delta^{k}_{\cF} = d_{\cF}d^{\ast}_{\cF} + d^{\ast}_{\cF}d_{\cF} : C^{\infty}(M,\Lambda^k\cF^{\ast}) \to C^{\infty}(M,\Lambda^k\cF^{\ast}).$$
\end{definition}

\begin{remark}
Of course, it is necessary to examine the self-adjointness of this Hodge-Laplace operator. One way to do this seems to generalise the Chernoff criterion \cite{Chernoff} to non-smooth vector bundles such as $\Lambda^{k}\cD^{\ast}$, using local presentations. We leave this for future work.
\end{remark}

\section{Isometries of distributions}\label{app:isometry}

Let $(M,\cD)$ and $(M^\prime,\cD^\prime)$ be smooth distributions equipped with Riemannian metrics. The purpose of this section is to give, in proposition \ref{prop:isometry}, the relation of the associated horizontal Laplacians via an isometry.

To this end, first let us define the notion of isometry between distributions. Consider a diffeomorphism $f : M\to M^\prime$. We have the induced map $f^*: C^\infty(M^\prime)\to C^\infty(M)$ given, for $u\in C^\infty(M^\prime)$, by 
\[
f^*u(x)=u[f(x)], \quad x\in M,
\]
and $f_*: \cX(M)\to \cX(M^{\prime})$ given, for $X\in \cX(M)$, by 
\[
f_* X(y)=df_{x}[X(x)] \quad \text{ for all } y=f(x)\in M^{\prime}.
\]
Recall that for any $\phi \in C^{\infty}(M)$ and $X \in \cX(M)$ we have $f_{\ast}(\phi \cdot X)= (\phi\circ f^{-1}) \cdot f_{\ast}(X)$. Therefore, if $f_{\ast}(\cD) \subseteq \cD^{\prime}$, then $f_{\ast}(I_{x}\cD) \subseteq I_{f(x)} \cD^{\prime}$ for every $x \in M$. So, in this case, for every $x \in M$, the map $f$ induces a linear map $f_{\ast} : \cD_{x} \to \cD^{\prime}_{f(x)}$.

\begin{definition}\label{dfn:isometry}
We say that $f$ is an isometry of distributions if  $f_*(\cD)= \cD^{\prime}$ and, for all $x\in M$ and for any $X,Y\in \cD_x$,
\begin{equation}\label{eqn:isomdfn}
\langle f_*(X), f_*(Y)\rangle_{f(x)}=\langle X, Y\rangle_{x}.
\end{equation}
\end{definition}
Let us make the following observations regarding definition \ref{dfn:isometry}.
\begin{enumerate}


\item If $f : M \to M'$ is a diffeomorphism such that $f_*(\cD)=\cD'$, then given a Riemannian metric $\langle\ ,\ \rangle_{\cD'}$ on $\cD^\prime$, there exists a Riemannian metric $\langle\ ,\ \rangle_{\cD}$ on $\cD$ such that $f$ is an isometry. We can use the relation \eqref{eqn:isomdfn} as a definition of $\langle\ ,\ \rangle_{\cD}$.

\item Let $f : M \to M'$ be a diffeomorphism such that $f_*(\cD)=\cD'$. Let $x \in M$ and take $U$ a neighborhood of $x$ in $M$ so that $\cD\left|_U\right.$ is generated by $X_1,\ldots, X_n \in \cD$ such that $\{[X_1]_x,\ldots,[X_n]_x\}$ is a basis of $\cD_x$. Let $(E_U, \rho_U)$ be the minimal local presentation of $(M,\cD)$ associated with this data, constructed as in \S \ref{sec:minlocpre}. Then there exists a minimal local presentation $(E_{f(U)},\rho_{f(U)})$ of $(M',\cD')$ such that $df\circ\rho_{U}=\rho_{f(U)}$. Indeed, put $E_{f(U)}$ the trivial bundle $f(U) \times \R^n$, and $\rho_{f(U)}(f(\zeta),\lambda_1,\ldots,\lambda_n)=\sum_{i=1}^n \lambda_i \cdot f^{\ast}X_i(f(\zeta))$ for every $\zeta \in U$.

\item Since $f$ is a diffeomorphism and $f_*(\cD)=\cD'$, every minimal local presentation of $(M',\cD')$ is $(E_{f(U)},\rho_{f(U)})$, where $(E_U, \rho_U)$ is a minimal local presentation of $(M,\cD)$.

\item 
Again, let $f : M \to M'$ be a diffeomorphism such that $f_*(\cD)=\cD'$ as in the previous item.
Let $(E_{U^\prime},\rho^\prime_{U^\prime})$ be a local presentation over some open subset $U^\prime\subset M^\prime$. Then there exist a local presentation $(E_{U},\rho_U)$ over the open subset $U=f^{-1}(U^\prime)\subset M$ and a isomorphism of vector bundles $\phi: E_U\to E_{U^\prime}$ over $f:U\to U^\prime$ such that the following diagram commutes: 
\begin{equation}\label{eqn:diag1}
\xymatrix{
E_{U} \ar@{->}[r]^{\phi} \ar@{->}[d]^{\rho_U} & E_{U^\prime} \ar@{->}[d]^{\rho^\prime_{U^\prime}} \\
TU \ar@{->}[r]^{df} \ar@{->}[d] & TU^\prime \ar@{->}[d] \\
U \ar@{->}[r]^{f} & U^\prime
}
\end{equation}
We put $E_{U}=f^*E_{U^\prime}$. Since $\phi$ and $df$ are  isomorphisms, we can easily find a map $\rho_U: E_U\to TU$ so that the diagram commutes and 
$ \rho_U(\Gamma_c E_U) = \cD\left|_U\right.$.     

\item Regarding the equivalence relation we discussed in \S  \ref{sec:equivpres}, let $(E_U, \rho_U)$ and $(E_V, \rho_V)$ be minimal local presentations of $(M,\cD)$ such that $U\cap V \neq 0$. Let $W$ be an open subset of $U \cap V$. As shown in the proof of Prop.  \ref{prop:equivminlocpr}, the pullback vector bundle $E_{f(U)} \times_{(A_{f(U),f(W)}, A_{f(V),f(W)})} E_{f(W)}$ defines an equivalence between the minimal local presentations $(E_{f(U)}, \rho_{f(U)})$ and $(E_{f(V)}, \rho_{f(V)})$ of $(M',\cD')$.


\item
For any diffeomorphism $f: M\to M^\prime$ such that $f_*(\cD)=\cD^\prime$, one can define the pull-back map
\[
f^* : C^\infty(M^\prime,\cD^{\prime*})\to C^\infty(M,\cD^*)
\]
by the formula
\[
f^*\omega^*(x)=f^*[\omega^*(f(x)], \quad x\in M.  
\]
for $\omega^* \in C^\infty(M^\prime,\cD^\prime)$, where $f^{\ast} : \cD^{\prime*}_{f(x)}\to \cD^*_{x}$ is the dual map of the map $f_{\ast} : \cD_{x} \to \cD^{\prime}_{f(x)}$. We show in the lemma below that it is well defined.
\end{enumerate}

\begin{lemma}\label{lem:itemd}
The map $f^{\ast}$ defined in the last item above maps $C^{\infty}$ sections to $C^{\infty}$ sections.
\end{lemma}
\begin{proof}
Let $(E_{U^\prime},\rho^\prime_{U^\prime})$ be a local presentation over some open subset $U^\prime\subset M^\prime$. As shown in item (d) above, there exist a local presentation $(E_{U},\rho_U)$ over the open subset $U=f^{-1}(U^\prime)\subset M$ and a morphism of vector bundles $\phi: E_U\to E_{U^\prime}$ over $f:U\to U^\prime$ such that diagram \ref{eqn:diag1} commutes. 
The induced map $\phi^\ast : C^\infty(U^\prime, E^\ast_{U^\prime})\to C^\infty(U, E^\ast_{U})$ is defined by 
\[
\phi^\ast s(x)=\phi^*_x[s(f(x))], \quad x\in U.
\]
We have the following commutative diagram
\begin{equation} 
\xymatrix{
C^\infty(U^\prime, \cD^{\prime*})  \ar@{->}[r]^{f^\ast} \ar@{->}[d]^{\hat\rho^{\prime*}_{U^\prime}} &  C^\infty(U, \cD^*) \ar@{->}[d]^{\hat\rho^*_{U}} \\  C^\infty(U^\prime, E^\ast_{U^\prime})
\ar@{->}[r]^{\phi^\ast} & C^\infty(U, E^\ast_{U}) 
}
\end{equation}
By this diagram, if $\omega^{\ast}_{U^\prime}\in C^\infty(U^\prime, E^\ast_{U^\prime})$ is the local realization of $\omega^{\ast}\in C^\infty(U^\prime,\cD^{\prime*})$, then the local realization of $f^*\omega^{\ast}$ is $\phi^*\omega^{\ast}_{U^\prime}\in C^\infty(U, E^\ast_{U})$. So  $f^*\omega^{\ast}$ is smooth. 
\end{proof}



The horizontal Laplacians associated with two distributions which are isometric are related in the way described by proposition \ref{prop:isometry} below:

\begin{prop}\label{prop:isometry}
Let $(M,\cD)$ and $(M^\prime,\cD^\prime)$ be smooth distributions equipped with Riemannian metrics and $\mu$ and $\mu^\prime$ smooth positive densities on $M$ and $M^\prime$, respectively. If $f: M\to M^\prime$ is an isometry of distributions and $f^*\mu=\mu^\prime$, then the pull-back map $f^*: C^\infty(M^\prime)\to C^\infty(M)$ commutes with $\Delta_{\cD}$:
\[
f^*\circ \Delta_{\cD^\prime}u=\Delta_{\cD}\circ f^*u,\quad u\in C^\infty(M^\prime). 
\]
\end{prop}

\begin{proof}
One can check that the following diagram commutes:
\begin{equation} 
\xymatrix{
\Omega^1(M^\prime) \ar@{->}[r]^{f^*} \ar@{->}[d]^{ev^*} & \Omega^1(M) \ar@{->}[d]^{ev^*} \\
C^\infty(M^\prime,\cD^{\prime*}) \ar@{->}[r]^{f^*} & C^\infty(M,\cD^*)
}
\end{equation}
Since the de Rham differential commutes with $f^*$, using the definition of $d_\cD$, we immediately get that $f^*$ commutes with $d_{\cD}$:
\begin{equation}\label{eq:commute}
f^*\circ d_{\cD^\prime}u=d_{\cD}\circ f^*u,\quad u\in C^\infty(M^\prime). 
\end{equation}
If $f$ is an isometry of distributions, for every $x \in M$, the induced map $f_{\ast} : \cD_{x} \to \cD^{\prime}_{f(x)}$ is an isometric isomorphism.
If, in addition, $f^*\mu=\mu^\prime$, then the pull-back maps $f^* : C^\infty(M^\prime)\to C^\infty(M)$ and $f^* : C^\infty(M^\prime,\cD^{\prime*})\to C^\infty(M,\cD^*)$ preserve the inner products and define unitary operators $f^* : L^2(M^\prime,\mu^\prime)\to L^2(M,\mu)$ and $f^* : L^2(M^\prime,\cD^{\prime*},\mu^\prime)\to C^\infty(M,\cD^*,\mu)$. Therefore, taking adjoints in \eqref{eq:commute}, we get that $f^*$ commutes with $d^*_{\cD}$: 
\[
f^*\circ d^*_{\cD^\prime}u=d^*_{\cD}\circ f^*u,\quad u\in C^\infty(M^\prime,\cD^{\prime*}). 
\]
The statement of the proposition then follows immediately.
\end{proof}


\bibliographystyle{habbrv} 

\end{document}